\documentclass[preprint,12pt]{elsarticle}
\usepackage[left=2.5cm,right=2.5cm,
top=3cm,bottom=3cm]{geometry}
\usepackage{amsmath,amssymb,amsthm,mathtools}
\usepackage{euscript,mathrsfs}
\usepackage{graphicx}
\usepackage{wrapfig}
\usepackage[dvipsnames]{xcolor}
\usepackage{wasysym}
\usepackage{verbatim}
\usepackage{empheq}
\usepackage{adjustbox}
\usepackage{xcolor}
\usepackage{mathrsfs}
\usepackage{hyperref}
\definecolor{urlcolor}{HTML}{990000}
\definecolor{linkcolor}{HTML}{005F5F}
\hypersetup{pdfstartview=FitH,  linkcolor=linkcolor,urlcolor=urlcolor, colorlinks=true,citecolor=blue}
\makeatletter
\def\@ssect@ltx#1#2#3#4#5#6[#7]#8{%
  \def\H@svsec{\phantomsection}%
  \@tempskipa #5\relax
  \@ifdim{\@tempskipa>\z@}{%
    \begingroup
      \interlinepenalty \@M
      #6{%
       \@ifundefined{@hangfroms@#1}{\@hang@froms}{\csname @hangfroms@#1\endcsname}%
       {\hskip#3\relax\H@svsec}{#8}%
      }%
      \@@par
    \endgroup
    \@ifundefined{#1smark}{\@gobble}{\csname #1smark\endcsname}{#7}%
  }{%
    \def\@svsechd{%
      #6{%
       \@ifundefined{@runin@tos@#1}{\@runin@tos}{\csname @runin@tos@#1\endcsname}%
       {\hskip#3\relax\H@svsec}{#8}%
      }%
      \@ifundefined{#1smark}{\@gobble}{\csname #1smark\endcsname}{#7}%
      \addcontentsline{toc}{#1}{\protect\numberline{}#8}%
    }%
  }%
  \@xsect{#5}%
}%
\makeatother

\theoremstyle{plain}
\newtheorem{theorem}{Theorem}
\newtheorem{definition}[theorem]{Definition}
\newtheorem{proposition}[theorem]{Proposition}

\newtheorem{example}[theorem]{Example}
\newtheorem{remark}[theorem]{Remark}

\newtheorem{problem}[theorem]{Problem}

\newtheorem{corollary}[theorem]{Corollary}

\renewcommand{\phi}{\varphi}
\renewcommand{\epsilon}{\varepsilon}
\DeclareMathOperator{\im}{\mathrm{Im}}
\DeclareMathOperator{\re}{\mathrm{Re}}

\DeclareMathOperator{\diag}{\mathrm{diag}}

\def\wt#1{\widetilde#1}
\def\a{\alpha}

\def\bam{\ba^{-1}}
\def\baml{\alpha^{-1}}
\def\difa{\a-\a^{-1}}
\def\zaa{(z-\a)(z-\bam)}
\def\zaal{(z-\a)(z-\baml)}
\def\zaaa{\a+\bam-2\cos\tau}

\numberwithin{equation}{section}

\numberwithin{equation}{section}

\journal{Journal of Differential Equations}

\begin{document}

\title{Dynamical systems on torus related to general Heun equations: phase-lock areas and constriction breaking}

\author[1,2,3]{\normalsize Artem Alexandrov\fnref{au1}}
\author[4,2,5]{\normalsize Alexey Glutsyuk\fnref{au2}}
\affiliation[1]{Institute for Information Transmission Problems, Moscow, 127994, Russia}
\affiliation[2]{HSE University, Moscow, Russia}
\affiliation[3]{Phystech School of Applied Mathematics and Computer Science, Moscow Institute of Physics and Technology, Dolgoprudny 141700, Russia}
\affiliation[4]{Higher School of Modern Mathematics, Moscow Institute of Physics and Technology, Dolgoprudny 141700, Russia}
\affiliation[5]{CNRS UMR 5669 (UMPA, ENS de Lyon), Lyon, France}
\fntext[au1]{aleksandrov.aa@phystech.edu}
\fntext[au2]{aglutsyu@ens-lyon.fr}

% \begin{highlights}
%     \item Two families of general and confluent Heun equations, which can be treated as families of dynamical systems on the torus, are proposed
%     \item Phase-locking phenomenon in the obtained family of dynamical system on the torus is analyzed both analytically and numerically
%     \item It is shown that the rotation number quantization effect is observed in this family, but the constrictions are absent
% \end{highlights}


\begin{abstract}
The overdamped Josephson junction in superconductivity theory can be modeled by the family of dynamical systems on the torus, which is known as the RSJ model. This family admits an equivalent description by a family of second-order differential equations: special double confluent Heun equations. In the present paper, we construct two new families of dynamical systems on torus that can be equivalently described by a family of general Heun equations (GHE), with four singular points, and confluent Heun equations, with three singular points. The first family, related to GHE, is a deformation of the RSJ model, which will be denoted by dRSJ. The phase-lock areas of a family of dynamical systems on the torus are those level subsets of the rotation number function that have nonempty interiors. It is known that for the RSJ model, the rotation number quantization effect occurs: phase-lock areas exist only for integer rotation number values. Moreover, each phase-lock area is a chain of domains separated by points. Those separation points that do not lie on the abscissa axis are called constrictions. In the present paper, we study phase-lock areas in the new family dRSJ. The quantization effect remains valid in this family. On the other hand, we show that in the new family dRSJ the constrictions break down. 
\end{abstract}

\maketitle

\newpage
\tableofcontents

\section{Introduction}

\subsection{Physical motivation and brief description of main results}

Heun equation in its different forms is ubiquitous in modern theoretical physics~\cite{hortaccsu2018heun}. This equation arises in many contexts: in quantum mechanics as the Schrodinger equation in a particular potential~\cite{ishkhanyan2016,figueiredo2024}, in the investigation of black hole quasinormal modes~\cite{fiziev2011application}, and in the description of conformal blocks~\cite{lisovyy2022perturbative}.

Despite the complicated structure of the Heun equation, it has many interesting properties (see~\cite{ronveaux1995heun,slavyanov-lay2000} for a comprehensive discussion) that allow one to introduce a suitable change of variables and see how the Heun equation transforms into the well-known and quite simple equations like the Lame equation, Whittaker-Hill, and Ince equations. From the physical point of view, the most interesting object directly related to the Heun equations is the monodromies around the singular points. In the context of high-energy physics, such monodromies are directly related to the quasinormal modes of black holes~\cite{hatsuda2020quasinormal,aminov2022black} and the properties of conformal blocks~\cite{litvinov2014classical}. The simplest but illustrative example of the significance of monodromy is that in the case of the double confluent Heun equation (DCHE) the monodromy around zero corresponds to the band-gap structure in a certain Schrodinger equation with periodic potential. In the classical context, it describes the phase lock areas of the \emph{classical} family of dynamical systems on a two-dimensional torus (see~\cite{alexandrov2025} for detailed discussion of this correspondence). 

Recall that to a family of dynamical systems on 2D torus corresponds the {\it rotation number function}, which assigns to given parameters the rotation number of the corresponding system; see the definition of rotation number in \cite{arn} and in the next subsection. The {\it phase-lock areas} are those level sets of the rotation number function in the parameter space that have non-empty interiors. In the case of DCHE the corresponding family of dynamical systems on the torus $\mathbb T^2=\mathbb R^2_{\theta,\tau}\slash2\pi\mathbb Z^2$ 
is the family 
\begin{equation}\label{eq:RSJ-model-rescaled}
    \frac{d\theta}{d\tau}=\frac{\cos\theta+B+A\sin\tau}{\omega}.
\end{equation}
The description of its phase lock areas in terms of the Heun equation monodromy is developed in~\cite{buchstaber2017monodromy}. The main motivation for this research is quite simple: the family~\eqref{eq:RSJ-model-rescaled} is given by the so-called RSJ model~\eqref{eq:RSJ-model}, which provides the phenomenological description of the Josephson junction (JJ) shunted by a resistance in the overdamped limit. The parameters of the family~\eqref{eq:RSJ-model-rescaled} are $(B,A)$; $B$ is the abscissa, $A$ is the ordinate, the frequency $\omega$ is fixed. Its phase-lock areas exist only for integer rotation number values, see \cite{buchstaber2010rotation}. They are ``garlands'', as shown in Figure 1. Each phase-lock area contains infinitely many so-called {\it constrictions}: those self-intersection points of its boundary that lie outside the abscissa axis.

The phase-locking phenomenon (existence of phase-lock areas) is a well-known fundamental property of families of non-linear dynamical systems \cite{pikovsky2001synchronization,arn}. It is familiar for the overdamped Josephson junction (JJ) driven by periodic external current, which is known as the RSJ model \cite{mccumber1968effect}. In addition, phase-locking is observed in many other systems: microparticle systems~\cite{mishra2025phase}, superconducting nanowires~\cite{dinsmore2008}, charge density waves~\cite{gruner1985charge}, skyrmions~\cite{reichhardt2015shapiro}.

Motivated by the mentioned facts, in the present paper we extend the idea that phase-locking in dynamical system family on 2D torus can be described in terms of the monodromy of DCHE. We introduce two new families~\eqref{dyn3},~\eqref{dyn4}  of dynamical systems on the 2D torus, which correspond, respectively, to the general Heun equations (GHE) and the confluent Heun equations (CHE). The family~\eqref{dyn3} is a deformation of the family~\eqref{eq:RSJ-model-rescaled}. We show that in family~\eqref{dyn3}, which corresponds to GHE, the phase-lock areas also exist only for integer rotation number values. However, compared to family~\eqref{eq:RSJ-model-rescaled}, which is given by the RSJ model and corresponds to DCHE, the phase-lock areas in the family corresponding to GHE do not contain constrictions.

\subsection{The RSJ model, rotation number and phase-lock areas}

\begin{figure}
    \centering
    \includegraphics[width=0.8\linewidth]{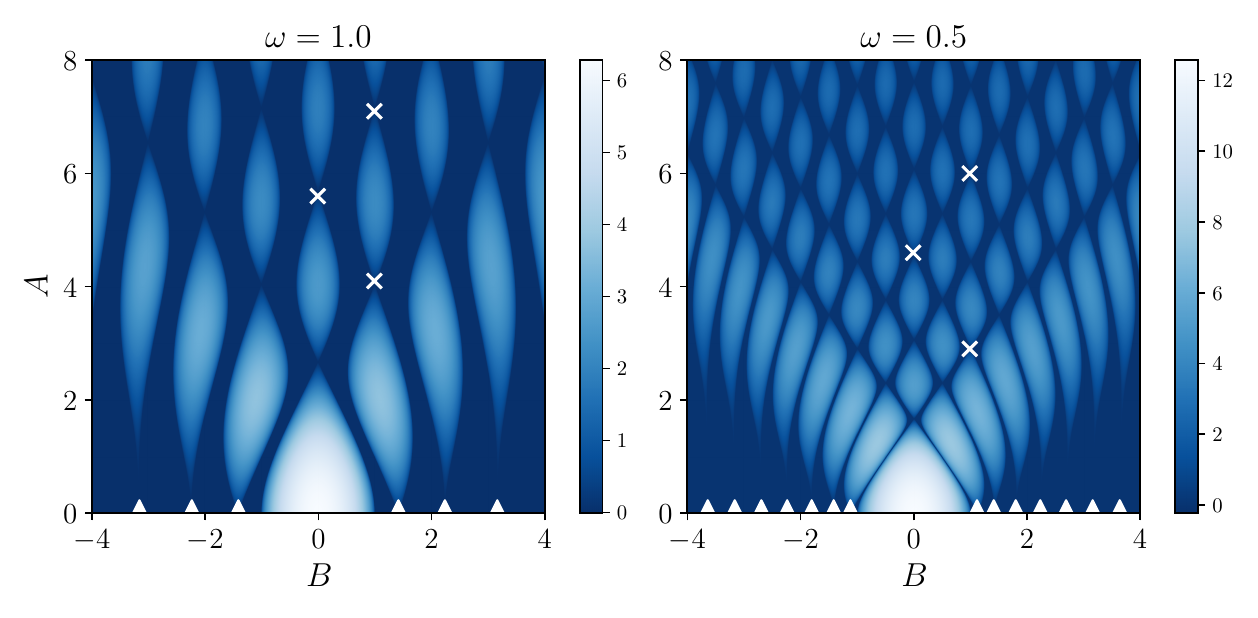}
    \caption{Phase-lock areas of the RSJ model~\eqref{eq:RSJ-model-rescaled} for different $\omega$. White crosses correspond to constriction points (only three constrictions are shown), white upper triangles correspond to growth points. Color bar represents the values of the largest Lyapunov exponent.}
    \label{fig:phase-lock-constr}
\end{figure}

For a brief introduction to the phase-locking phenomenon, it suffices to consider the RSJ model of the overdamped Josephson junction (JJ). It has the following equation of motion:
\begin{equation}\label{eq:RSJ-model}
    \frac{d\phi}{dt} = -\sin\phi + B + A\cos\omega t.
\end{equation}
Here $\phi$ is the phase difference between two superconductors with insulator link between them, $B$ and $A$ are dimensionless amplitudes of DC and AC currents applied to the junction, $\omega$ is the rescaled frequency of the AC current.

For further discussion, it is convenient to introduce the dimensionless time $\tau=\omega t$ and then shift it by $\pi/2$ and introduce the shifted phase $\theta=\phi+\pi/2$, which gives us the dynamical system family~\eqref{eq:RSJ-model-rescaled} on the torus $\mathbb T^2=\mathbb R^2_{\theta,\tau}\slash2\pi\mathbb Z^2$:  
\begin{equation*}
    \frac{d\theta}{d\tau}=\frac{\cos\theta+B+A\sin\tau}{\omega}.
\end{equation*}
 By definition, the \emph{rotation number} of the system~\eqref{eq:RSJ-model-rescaled}, see \cite{arn}, is given by
\begin{equation*}
    \rho=\lim\limits_{T\rightarrow+\infty}\frac{\theta(T)}{T}\in\mathbb{R}.
\end{equation*}
Here $\theta(\tau)$ is an arbitrary solution of equation \eqref{eq:RSJ-model-rescaled}; the rotation number is independent of its choice, see \cite{arn}. 
In families of dynamical systems on the 2D torus, typically each rational rotation number value $\rho$ corresponds to a phase-lock area. See a more precise result in \cite{glutryb}. But for the system~\eqref{eq:RSJ-model-rescaled} it is known that \emph{rotation number quantization effect} takes place: phase locking occurs only for \emph{integer} values of the rotation number~\cite{buchstaber2010rotation}. Omitting the details, it happens because the Poincar\'{e} first return map $h$ of the system~\eqref{eq:RSJ-model-rescaled}, which in our case is the time $2\pi$ flow self-map of the coordinate $\theta$-circle, is a M\"{o}bius transformation that maps the unit circle to the unit circle. The M\"{o}bius transformations can be classified in terms of conjugacy classes of the $\mathrm{SL}(2,\mathbb{R})$ group. From a physical point of view, this rotation number $\rho$ is equal up to a numerical factor (that contains Planck constant) to the long time averaged voltage on JJ, driven by the external current with dimensionless DC part $B$ and dimensionless AC part $A\cos\omega t$ (see~\cite{renne1974some,waldram1982alternative} for the physical interpretation). If one fixes $A$ and $\omega$ and studies the rotation number $\rho$ as a function of $B$, so-called Cantor staircase will appear. For the RSJ model this staircase is known as \emph{Shapiro steps}, see~\cite{shapiro1963josephson}.

In terms of this M\"{o}bius transformation, we effectively deal with the \emph{discrete} dynamical system on the unit circle $S^1$. Now, consider the trajectory of~\eqref{eq:RSJ-model-rescaled} with a given initial condition $\theta(0)=\theta_0$. If the M\"{o}bius transformation has an attractive fixed point on the unit circle $S^1$, then the Lyapunov exponent $\Lambda$ of the corresponding discrete dynamical system will be negative. In fact, if the M\"{o}bius transformation is hyperbolic, then it has two fixed points on $S^1$:  one of them is attractive; the other is repulsive. The corresponding Lyapunov exponents are $\Lambda_{-}<0$ and $\Lambda_{+}>0$ and $\Lambda_{+}=-\Lambda_{-}$. These Lyapunov exponents are logarithms of the multipliers of the M\"{o}bius  transformation at its fixed points. The interiors of the phase-lock areas correspond to non-zero values of the Lyapunov exponent $\Lambda_{+}$ (or $\Lambda_{-}$) (the detailed explanation is given in Appendix B of~\cite{alexandrov2025}). At the Fig.~\ref{fig:phase-lock-constr}, the phase-lock areas of~\eqref{eq:RSJ-model-rescaled} are shown. They are colored white. This picture has two main features. There exist points in the $(B,A)$ plane, lying outside the abscissa $B$-axis, where the intersections of phase-lock areas with horizontal lines shrink to the point; they are called \emph{constrictions}.  In the phase-lock areas corresponding to non-zero integer rotation numbers there exist similar self-intersection points in the abscissa axis $A=0$; they  are called \emph{growth points}. Infinite sequences of constrictions in the phase-lock areas are seen both in experiments and numerical simulations, for instance, see \cite[fig. 3]{panghotra2020giant} (experiment and numerical simulation) and \cite[p. 193, fig. 11.4]{lich-rus}, 
\cite[p. 88, fig. 5.2]{likh-ulr},  \cite[p. 339, fig. 11.4]{lich} (which refers to the paper \cite{ls}). Their existence was rigorously proved in a joint paper by A.V. Klimenko and O.L. Romaskevich~\cite{klim-rom}. Properties of constrictions are investigated in~\cite{glutsyuk2014adjacency,glutsyuk2019constrictions,glutsyuk2023-germs,bibilo2022families}. It was proved in \cite{bibilo2022families} that for every integer value $r$ of the rotation number all the constrictions in the corresponding phase-lock area lie in one vertical line with abscissa $B=r\omega$. This constriction alignment phenomena can be also seen in the above-mentioned pictures in physics books. It was conjectured in \cite{glutsyuk2014adjacency}, where partial positive results were obtained. 

The RSJ model~\eqref{eq:RSJ-model-rescaled} can be rewritten in terms of DCHE (\ref{heun}), as shown in~\cite{buchstaber2017monodromy,buchstaber2013explicit,buchstaber2015holomorphic}. In terms of DCHE, the mentioned M\"{o}bius transformation is conjugated to the projectivized monodromy of the DCHE matrix solution around zero. Inspired by this correspondence, in the present paper we introduce two new families of dynamical systems on torus. These two families admit equivalent description by, respectively, GHE and CHE. The first family, which corresponds to GHE, is family~\eqref{dyn3}. It can be considered as a deformation of the family~\eqref{eq:RSJ-model-rescaled}. The second family, which corresponds to CHE, is family~\eqref{dyn4}. The  families of Heun equations corresponding to families (\ref{dyn3}) and (\ref{dyn4}) are  (\ref{heunt3}) and (\ref{heunc4}) respectively.

\subsection{Plan of the paper}\label{subsec:main-results}

The two above-mentioned new families~\eqref{dyn3},~\eqref{dyn4} of dynamical systems corresponding to GHE and CHE are introduced in~\autoref{subsec:gen-non-confl} and~\autoref{subsec:confl-case} respectively, together with the corresponding Heun equations. The rotation number quantization effect in family (\ref{dyn3}) is proved in~\autoref{subsec:gen-non-confl}. The results stating that the dynamical systems in question are indeed equivalent to the corresponding Heun equations are: Theorem~\ref{heuntor} and Proposition~\ref{prodyn1} for family~\eqref{dyn3} and the GHE; Theorem \ref{thc} and Proposition~\ref{prodyn2} for family~\eqref{dyn4} and the CHE. Theorems~\ref{heuntor} and~\ref{thc} are stated and proved, respectively, in~\autoref{sec:lin-sys-Heun} and~\autoref{sec:confluent-case}. Propositions~\ref{prodyn1} and~\ref{prodyn2} are stated and proved, respectively, in~\autoref{subsec:four-sing-proof} and~\autoref{subsec:three-sings}. 

In~\autoref{sec:lin-sys-Heun} we introduce the general class of linear systems of differential equations on the Riemann sphere on vector function $Y=(Y_1(z),\,Y_2(z))$ with four Fuchsian singularities $0$, $\a$, $\beta$, $\infty$ corresponding to dynamical systems on the torus, namely, the so-called torus dynamical type linear systems. This means that their projectivizations, i.e., Riccati equations, are complexifications of dynamical systems on 2D torus: on the product of the unit circles in time and space Riemann spheres. In this case 
$\beta=\overline{\a}^{-1}$, and we rescale the variable so that $\a>0$. Proposition \ref{psym} gives a criterion for being a system of torus dynamical type in terms of residue matrix relations. We deal with a special subclass of the latter torus dynamical type systems, with the residue matrix at the origin being diagonal with zero second eigenvalue. This subclass is described by Proposition \ref{ptdiag}. The above-mentioned Theorem \ref{heuntor} describes explicitly those of the latter systems that are equivalent to the general Heun equation (GHE) on the second component
$E=Y_2(z)$. Theorem~\ref{thc} stated and proved in~\autoref{sec:confluent-case} is its analogue for confluent Heun equations (CHE). 

In~\autoref{subsec:four-sing-proof} and~\autoref{subsec:three-sings} we deduce dynamical systems~\eqref{dyn3} and~\eqref{dyn4} respectively from projectivizations of the above-mentioned linear systems. In~\autoref{subsec:dRSJ-family} we state and prove Theorem \ref{montr} and its Corollary \ref{cmontr} saying that in family (\ref{dyn3}) there are no constrictions. And also Theorem \ref{growd3}, which deals with the family (\ref{dyn3}) restricted to the hyperplane $\{A=0\}$ equipped with coordinates $(B,D)$. It describes analogues of the growth points of phase-lock areas lying in the abscissa $B$-axis. The proposition \ref{crot} stated and proved in the same place gives the formula for the rotation number of the family restricted to the $B$-axis. 

Preliminaries on the Heun equations are given in~\autoref{sec:lin-sys-Heun}. In~\autoref{subsec:open-problems} we present a list of open problems. Another open problem on torus dynamical type linear systems, is presented at the end of the preambula to~\autoref{sec:lin-sys-Heun}.

\section{From Heun equations to dynamical system on torus}\label{sec:Heun-to-torus}

\subsection{Generic non-confluent case}\label{subsec:gen-non-confl}

The corresponding family of dynamical systems is 
\begin{equation}\label{dyn3}
    \frac{d\theta}{d\tau}=\frac{\cos\theta+B+A\sin\tau}{\omega(1-\delta\cos\tau)}+D, \ \  \omega\in\mathbb R\setminus\{0\}, \ \delta\in(0,1).
\end{equation}
Note that in the case, when $\delta=0$, this system coincides with the RSJ model~\eqref{eq:RSJ-model-rescaled}. Keeping in mind this fact, hereafter we call the family~\eqref{dyn3} as \emph{deformed RSJ model of GHE type} (for brevity, we simply use denotation dRSJ).

The variable change
\begin{equation*}
    \Phi=e^{i\theta},\quad z=e^{i\tau}
\end{equation*}
and extension to complex domain transforms family~\eqref{dyn3} to family of Riccati equations
~\eqref{ric1}:
\begin{equation}\Phi'=-\frac{b(\difa)}{\zaal}(1+\Phi^2)-\left(\frac{\bar\nu+c}{z-\a}-\frac{\nu+c}{z-\baml}+
\frac{\nu}z\right)\Phi, \ \ \a>0, \ b, c\in\mathbb{R}, \ b\neq0,\label{ric3}\end{equation}
$$\omega=\frac{\a+\baml}{2b(\difa)}, \ \ \  \ D=-\re\nu, \ \ \ \ \delta=\frac2{\a+\baml}\in(0,1),$$
\begin{equation} 
B=\frac{c+\re\nu}{2b}, \ \ \ \ \ \ \  A=\frac{\im\nu}{b(\difa)}.\label{coefd3}\end{equation}
Riccati equations~\eqref{ric3} are projectivizations of 
linear systems of type~\eqref{kd}:
\begin{equation} Y'=\left(\frac{\left(\begin{matrix}\nu & 0\\ 0 & 0\end{matrix}\right)}z+
\frac{\left(\begin{matrix}\phi & b\\ -b & \phi-\bar\nu-c\end{matrix}\right)}{z-\alpha}  + \frac{\left(\begin{matrix}\psi-\nu-c & 
-b\\ b & \psi\end{matrix}\right)}{z-\alpha^{-1}}\right)Y, \ \ \a\in\mathbb{R}_+, \   b\in\mathbb{R}_{\geq0}, \   c\in\mathbb{R}.  \label{kd3}\end{equation}
Namely, $\Phi(z)$ is a solution of the Riccati equation~\eqref{ric1}, if and only if $\Phi(z)=\frac{Y_2(z)}{Y_1(z)}$, where
$Y=(Y_1(z), Y_2(z))$ is a solution of~\eqref{kd3}. 
Riccati equation~\eqref{ric3}, and hence~\eqref{dyn3}, are independent on the diagonal parameters $\phi$, $\psi$ of a linear system, which may be chosen arbitrarily. One can choose them in four possible ways so that the system~\eqref{kd} is equivalent to a Heun equation on the second component $E=Y_2(z)$ of the vector function $Y=(Y_1,Y_2)$. Namely, this holds if and only if $\phi$, $\psi$ are chosen to satisfy the quadratic equations~\eqref{psi}: 
\begin{equation}b^2=\phi(\bar\nu+c-\phi)=\psi(\nu+c-\psi)\in\mathbb{R}_{\geq0}; \ \ \text{ that is,  either } \ \psi=\bar\phi, \ \ \text{ or } \ \psi=\nu+c-\bar\phi.\label{psi3}\end{equation}
The corresponding Heun equation is~\eqref{heunt}: 
\begin{equation}\label{heunt3}
    z(z-\baml)(z-\a)E''+(-\nu(z-\alpha)(z-\baml)+q z(z-\baml)+sz(z-\a))E'+(uz+d)E=0,
\end{equation}
    $$q=\bar\nu+c+1-2\phi, \ \ s=\nu+c+1-2\psi, \ u=(\bar\nu+c-\phi-\psi)(c+1-\phi-\psi), \ \ d=\nu((c+\bar\nu-\phi)\baml-\psi\a);$$
$$s=\begin{cases} \bar q, \ \text{ if } \ \psi=\bar\phi\\
s=2-\bar q, \ \text{ if } \psi=\nu+c-\bar\phi.\end{cases}$$
\begin{remark} For a given pair of parameters $(\phi,\psi)=(\phi_0,\psi_0)$ satisfying~\eqref{psi3} the Heun equations~\eqref{heunt3} corresponding to the three other parameter pairs 
    $$(\phi_1,\psi_1)=(\phi_0, \nu+c-\psi_0), \  (\phi_2,\psi_2)=(\bar\nu+c-\phi_0,\ \psi_0), \ 
(\phi_3,\psi_3)=(\bar\nu+c-\phi_0,\ \nu+c-\psi_0)$$
satisfying~\eqref{psi3} are equations on functions $E_j(z)$, $j=1,2,3$,  
obtained from the Heun equation on $E(z)$ corresponding to $(\phi_0,\psi_0)$ by 
gauge transformations 
    $$E_j(z)=(z-\a)^{\phi_j-\phi_0}(z-\baml)^{\psi_j-\psi_0}E(z).$$
\end{remark}

\begin{theorem} In family~\eqref{dyn3} the Rotation Number Quantization Effect from \cite{buchstaber2010rotation} holds: phase-lock areas exist only for integer rotation number values. 
\end{theorem}

\begin{proof} The family is strictly increasing in $B$. It is a projectivization of family of linear systems, and 
hence, has M\"obius Poincar\'e map. It is known, see \cite{buchstaber2010rotation}, that the Rotation Number Quantization Effect holds in any family of M\"obius circle diffeomorphisms that is  strictly monotonous in some parameter. This implies the statement of the theorem. 
\end{proof}

\subsection{Confluent case}\label{subsec:confl-case}

The corresponding family of dynamical systems is
\begin{equation}\label{dyn4}
    \frac{d\theta}{d\tau}=\frac{\cos\theta+B+A\sin\tau}{\omega(1-\cos\tau)}+D,
\end{equation}
The variable change $\Phi=e^{i\theta}$, $z=e^{i\tau}$ 
and extention to complex domain transforms family~\eqref{dyn3} to family of Riccati equations~\eqref{ric2}: 
\begin{equation}\label{ric4}
    \Phi'=-\frac{b}{(z-1)^2}(1+\Phi^2)+\left(-\frac{\nu}{z}-\frac{g}{(z-1)^2}+\frac{\nu-\bar\nu}{z-1}\right)\Phi, \ \  \ b\in\mathbb R\setminus\{0\}, \ g\in\mathbb{R},
\end{equation}
\begin{equation}\label{confcs4}
    \omega=\frac{1}{b}, \ \ B=\frac{g}{2b}, \ \ A=\frac{\im\nu}{b}, \ \ D=-\re\nu.
\end{equation}
Riccati equations~\eqref{ric4} are projectivizations of linear systems 
~\eqref{y'c}: 
\begin{equation} Y'=\left(\frac{\left(\begin{matrix}\nu & 0\\ 0 & 0\end{matrix}\right)}z+
\frac{\left(\begin{matrix}a_2 & b\\ - b & a_2-g\end{matrix}\right)}{(z-1)^2}  + \frac{\left(\begin{matrix} 
a_1 & 0\\ 0 & a_1+\nu-\bar\nu\end{matrix}\right)}{z-1}\right)Y, \ \ \ b\in\mathbb{R}\setminus\{0\}, \ \ g\in\mathbb{R}:\label{y'c4}\end{equation}
$\Phi(z)$ is a solution of Riccati equation~\eqref{ric4}, if and only if $\Phi(z)=\frac{Y_2(z)}{Y_1(z)}$, where 
$Y=(Y_1(z), Y_2(z))$ is a solution of~\eqref{y'c4}. Riccati equation~\eqref{ric4}, and hence,~\eqref{dyn4}, are independent on the diagonal parameters $a_1,a_2\in\mathbb{C}$ of linear system, which may be chosen arbitrarily. 
Generically, one can choose them in two possible ways in order that system~\eqref{y'c4} be equivalent to a Heun equation on the second component 
 $E=Y_2(z)$ of vector function $Y=(Y_1,Y_2)$. Namely, this holds if and only if
\begin{equation} b^2=a_2(g-a_2), \ \ \ a_2(\nu-\bar\nu)+a_1(2a_2-g)=0.\label{ech4}\end{equation}
Then the corresponding renormalized confluent Heun equation is~\eqref{heunc}: 
\begin{equation}z(z-1)^2E''+(pz(z-1)+qz+s)E'+(uz+d)E=0,\label{heunc4}\end{equation}
$$p=\bar\nu+2-2\nu-2a_1, \ \ q=\nu+g-2a_2, \ \ s=-\nu,$$
 \begin{equation} u=(a_1+\nu-1)(a_1+\nu-\bar\nu), \ \ d=\nu(a_2-g-a_1-\nu+\bar\nu).\label{coefl4}
 \end{equation}
\begin{remark} 
System of equations~\eqref{ech4} in $(a_1,a_2)\in\mathbb{C}^2$ has a solution, if and only if 

- either $b\neq\pm\frac g2$; then it has two different 
solutions;

- or $b=\pm\frac g2$ and $\frac g2(\nu-\bar\nu)=0$; then $a_2=\frac g2$ is the  unique, double root of the first equation and 
each pair $(a_1,a_2)$ with $a_2=\frac g2$ and arbitrary $a_1\in\mathbb{C}$ is a solution. 

In the case, when $b=\pm\frac g2$ and $\frac g2(\nu-\bar\nu)\neq0$, system~\eqref{ech4} has no solution $(a_1,a_2)$. 

For  given two solutions $(a_1,a_2)$, $(\wt a_1,\wt a_2)$ of~\eqref{ech4} the corresponding confluent 
Heun equations on functions $E$ and $\wt E$ are obtained one from the other by gauge transformation
$$\wt E(z)=(z-1)^{\wt a_1-a_1}e^{\frac{a_2-\wt a_2}{z-1}}E(z).$$
\end{remark}

\subsection{Open problems}\label{subsec:open-problems}

Consider the dRSJ family of dynamical systems~\eqref{dyn3} with fixed $\omega$, $\delta$, $D$ and variable parameters $(B,A)\in\mathbb{R}^2$. 

\begin{problem} Study the portrait of phase-lock areas in $\mathbb{R}^2_{B,A}$ of thus obtained two-parameter family of 
dynamical systems. For fixed $\delta\in[0,1)$ and $D\in\mathbb{R}$ study asymptotics of the phase-lock area portrait, as $\omega\to0$. This problem is open for $\delta=D=0$ as well.
\end{problem}

\begin{problem} Study asymptotics of the rotation number function $\rho(B,A)$ in family~\eqref{dyn3} with fixed $D$, as $\delta\to1$, i.e., in the confluent limit. Find the first and second main asymptotic terms. 
\end{problem}

\begin{problem} Describe those Heun equations~\eqref{heunt3},~\eqref{heunc4} related to dynamical systems~\eqref{dyn3},~\eqref{dyn4} 
that have 

- either polynomial, or rational solutions;

- or  meromorphic solutions on $\overline{\mathbb{C}}\setminus\{1\}$ (with possible essential singularity at $1$). 
\end{problem}

\begin{problem} Study dynamical systems on torus corresponding to isomonodromic families of general  Heun equations: foliation by isomonodromic leaves in the space of parameters of dynamical systems. Study asymptotic behavior of leaves. Apparently this should be related to properties of solutions of Painlev\'e 6 equations.
\end{problem}

\section{Linear systems and Heun equations}\label{sec:lin-sys-Heun}

Recall that the family of Heun equations with four distinct singularities $0$, $\alpha$, $\beta$, $\infty$ is the following family of second-order linear differential equations:
\begin{equation} z(z-\alpha)(z-\beta)E''+(p(z-\alpha)(z-\beta)+qz(z-\beta)+sz(z-\alpha))E'+(uz+d)E=0.
\label{heung}\end{equation}

The family of double confluent Heun equations is obtained from~\eqref{heung} by passing to the limit, as $\alpha\to\infty$, $\beta\to0$. It is the family 
\begin{equation} z^2E''+(-\mu z^2+cz+t)E'+(-az+\lambda)E=0\label{dheun}\end{equation}
\begin{remark} Family~\eqref{heung} depends on 7 parameters: $\alpha$, $\beta$, $c$, $d$, $u$, $v$, $\gamma$. 
But one of them, namely $\beta$, can be normalized to be equal to $\alpha^{-1}$  by rescaling the variable $z$. Thus, there are six essential parameters in~\eqref{heung}. 
Similarly,  in~\eqref{dheun} the coefficient $\mu$ can be normalized to be equal to one by rescaling the variable $z$, so there are four essential parameters in~\eqref{dheun}. All the parameters in question are complex. 
There are also different types of confluent Heun equations, when only two singularities (or three, or four singularities) collide to one singularity. These types are called respectively confluent, biconfluent and triconfluent Heun equations, respectively, see  
\cite{ronveaux1995heun, slavyanov-lay2000}.\end{remark}

Each Heun equation with four distinct singularities can be written as a two-dimensional Fuchsian linear system
\begin{equation}Y'=\left(\frac Kz+\frac{R_1}{z-\alpha}+\frac{R_2}{z-\beta}\right)Y; \ \ 
\ \ K, \ R_1, \ R_2 \ \text{ are 2x2-matrices.}\label{systh}\end{equation}
We study its projectivization, which is a Riccati equation on a $\overline{\mathbb{C}}$-valued function $\Phi(z)$: a holomorphic 
foliation on $\overline{\mathbb{C}}_{\Phi}\times\overline{\mathbb{C}}_z$ induced by~\eqref{systh} via tautological projection 
$$\mathbb{C}^2_{Y_1,Y_2}\setminus\{0\}\to\mathbb{CP}^1_{[Y_1:Y_2]}=\overline{\mathbb{C}}_{\Phi}, \ \ \Phi=\frac{Y_2}{Y_1}.$$  
The latter Riccati equation has following property: all its solutions $\Phi(z)$ have the form $\Phi(z)=\frac{Y_2(z)}{Y_1(z)}$, where $Y=(Y_1(z), Y_2(z))$ is a vector solution of system~\eqref{systh}. 

\begin{definition} 

A meromorphic linear system of first order differential equations on $\overline{\mathbb{C}}_z$ on a vector function $Y(z)=(Y_1(z),Y_2(z))$  is said to have {\bf torus dynamical type,} if the corresponding Riccati equation in the function $\Phi(z)=\frac{Y_2(z)}{Y_1(z)}$  has invariant torus $\mathbb{T}^2:=S^1_\Phi\times S^1_z$, $S^1_w:=\{|w|=1\}\subset\overline{\mathbb{C}}$. This means that for every $(\Phi_0,z_0)\in\mathbb{T}^2$ the solution $\Phi(z)$ of the Riccati equation with the initial condition $\Phi(z_0)=\Phi_0$ takes values in the unit circle $S_\Phi^1$ along the unit circle $S_z^1$. 

A \emph{scalar linear second order differential equation} is said to be of {\it torus dynamical type}, if there exists a linear system of torus dynamical type such  that for every vector solution $Y(z)=(Y_1(z),Y_2(z))$ of the system, the component $Y_2(z)$ is a solution of the scalar equation.

\end{definition}

\begin{example} Consider linear systems of type 
\begin{equation}\label{tty}
    Y'=\left(\frac{\diag(-\frac s2,0)}{z^2}+\frac{\mathcal{B}}z+\diag(0, \frac s2)\right)Y, \ \ \ 
    \mathcal{B}=\left(\begin{matrix} -\ell &-\frac a2\\ \frac a2 & 0\end{matrix}\right), \ \ s,a\in\mathbb{R}.
\end{equation}
Their family yields an equivalent description of model of overdamped Josephson junction, see~\cite{bkt1,buchstaber2010rotation,buchstaber2015holomorphic, Foote, foott,  IRF}, \cite[subsection 3.2]{bg}.  It is known to be of the torus dynamical type and is equivalent to the following family of special double-confluent Heun equations.
\begin{equation}\label{heun}
    z^2E''+((\ell+1)z+\mu(1-z^2))E'+(\lambda-\mu(\ell+1)z)E=0,
\end{equation}
$$\mu:=\frac s2=\frac A{2\omega}, \ \ \lambda:=\frac{a^2-s^2}4.$$
Namely, for every solution $Y=(Y_1(z),Y_2(z))$ of system~\eqref{tty} the function $E=Y_2(z)$ satisfies~\eqref{heun} and vice versa: each solution of Heun equation~\eqref{heun} generates a solution of system~\eqref{tty}. See an implicit equivalent statement in \cite[subsection 3.2, p. 3869]{bg}. The relation of model of overdamped Josephson junction to special double confluent Heun equations was first observed  in \cite{tert2} and in \cite{bt0}. 
\end{example}

In subsection 3.1 we describe torus dynamical type Fuchsian linear systems with four singularities. We also show that such a system satisfying a genericity assumption is constant gauge equivalent to another torus dynamical type Fuchsian system with either upper-, or lower- triangular matrix $K$. 

We  also study the following problem. 

\begin{problem} Describe those torus dynamical type Fuchsian systems with four singularities on a vector-function 
$Y=(Y_1, Y_2)$ on $\overline{\mathbb{C}}$ 
 that are equivalent to Heun equations on $Y_2(z)$.  
That is, describe those 
Heun equations with four singularities that have torus dynamical type.
\end{problem}

In subsection 3.2 we present a first result towards its solution, treating the particular case, when the matrix 
$K$ is diagonal; we then renormalize $K$ to have zero second eigenvalue  multiplying the vector function $Y$ by a power of $z$. Our result yields an explicit  family~\eqref{heunt} 
of torus dynamical type Heun equations depending on five real parameters.

\subsection{Linear systems  of torus dynamical type with four singularities}

The starting point of their description is the following remark. 

\begin{remark} \label{remsym} A linear system has torus dynamical type, if and only if the corresponding Riccati equation 
is symmetric with respect to the map 
$(\Phi,z)\mapsto(\overline\Phi^{-1},\bar z^{-1})$. This is equivalent to the statement saying that the transformation 
\begin{equation} J:(Y_1,Y_2,z)\mapsto(\overline Y_2,\overline Y_1, \bar z^{-1})\label{barsym}
\end{equation}
preserves the linear system in question up to addition of a scalar matrix function. In particular, the singularity 
collection of a torus dynamical type linear system should be symmetric with respect to the unit circle in $\overline{\mathbb{C}}_z$. 
\end{remark}

Thus, we consider that the symmetric points $0$ and $\infty$ are singularities, 
and the other singularity pair $\alpha$, $\beta$ is also symmetric: $\beta=\bar\alpha^{-1}$. Then all the four singularities lie on the same line. 
Without loss of generality we consider this is a real line, hence 
$$\beta=\alpha^{-1} \ \text{ and } \ \alpha>0,$$
applying a rotation of the coordinate $z$. 

Everywhere below for a $2\times2$-matrix $K$ by $K^{tt}$ we denote the matrix obtained from $K$ by 
conjugation by the permutation matrix 
$$P:=\left(\begin{matrix} 0 & 1 \\ 1 & 0\end{matrix}\right).$$
Or equivalently, the matrix $K^{tt}$ is obtained from $K$ by permuting its diagonal terms and permuting 
its off-diagonal terms.

\begin{proposition} \label{psym} A linear system 
\begin{equation}Y'=\left(\frac{K}z +\frac{R_1}{z-\alpha}+\frac{R_2}{z-\alpha^{-1}}\right)Y\label{syst}, \ \ \ \ \ \alpha\in\mathbb R\setminus\{0\}
\end{equation}
is of torus dynamical type, if and only if the following matrix equalities hold  modulo $\mathbb{C} Id$, i.e., 
modulo scalar matrices, multiples of the identity:
\begin{equation}R_2=\overline R_1^{tt}, \ \ \ K+\overline K^{tt}+R_1+R_2=0.
\label{symrel}\end{equation}
\end{proposition}
\begin{proof} Let us find the image of system~\eqref{syst} under the involution $J$ given by~\eqref{barsym}. 
The transformation $(Y_1,Y_2,z)\mapsto(Y_2,Y_1,z)$ replaces the matrices $K$, $R_j$ in~\eqref{syst} by 
$K^{tt}$, $R_j^{tt}$ respectively. Replacing the vector function values and $z$ by their complex conjugates, 
we set $w:=\bar z$, $\wt Y_j:=\overline Y_{2-j}$,  
replaces the matrices of system by their complex conjugates, since the singularities $\alpha^{\pm1}$ are real. This  
yields system 
\begin{equation}\frac{d\wt Y}{dw}=\left(\frac{\overline K^{tt}}w +\frac{\overline R_1^{tt}}{w-\alpha}+\frac{\overline R_2^{tt}}{w-\alpha^{-1}}\right)\wt Y, \ \ \ \wt Y=(\wt Y_1,\wt Y_2)=(\overline Y_2, \overline Y_1).\label{syst2}
\end{equation}
Changing the variable $w$ to $\wt z:=w^{-1}$ yields 
$$\frac{d\wt Y}{d\wt z}=-\frac1{\wt z^2}\left(\wt z \overline K^{tt} +\frac{\wt z\overline R_1^{tt}}{1-\alpha\wt z}+\frac{\wt z\overline R_2^{tt}}{1-\alpha^{-1}\wt z}\right)\wt Y.$$
Simplifying the latter right-hand side with decomposition of ratios into elementary ones yields
\begin{equation}\frac{d\wt Y}{d\wt z}=\left(-\frac{\overline K^{tt}+\overline R_1^{tt}+\overline R_2^{tt}}{\wt z} +\frac{\overline R_2^{tt}}{\wt z-\alpha}+\frac{\overline R_1^{tt}}{\wt z-\alpha^{-1}}\right)\wt Y.\label{syst3}\end{equation}
This is the system obtained from~\eqref{syst} by the involution $J$. Together with Remark \ref{remsym}, this 
implies that for a system~\eqref{syst} being of torus dynamical type is equivalent to the statement saying 
that relations~\eqref{symrel} hold modulo $\mathbb{C} Id$. The proposition is proved.
\end{proof}
\begin{remark}
There is a class of constant gauge transformations $\wt Y=HY$, $H\in GL_2(\mathbb{C})$, that preserves 
the class of torus dynamical type linear systems. Namely, this holds if and only if $H\in U(1,1)$ up to scalar factor, i.e., if 
the operator $H$ preserves the cone $\{|Y_1|=|Y_2|\}$. Two systems one obtained from the other by a gauge 
transformation of the above type will be called {\it $U(1,1)$-gauge equivalent.}
\end{remark}
 
\begin{proposition} Every torus dynamical type linear system~\eqref{syst} such that 
\begin{equation} \text{the matrix } K \text{ has no eigenvector } v=(v_1,v_2) \text{ with } |v_1|=|v_2|\label{eigen}
\end{equation}
is $U(1,1)$-gauge equivalent to a system of torus dynamical type with 
the matrix $K$ being either upper, or lower-triangular. 
\end{proposition}
\begin{proof} Consider the tautological projection $\mathbb{C}^2\setminus\{0\}\to\mathbb{CP}^1_{[Y_1:Y_2]}=\overline{\mathbb{C}}_w$, 
$w=\frac{v_2}{v_1}$. The matrix $K$ has at least one eigenline. It is projected to a point $w_0\in\overline{\mathbb{C}}$ 
 that does not lie in the unit circle $\{|w|=1\}$. 
Applying a M\"obius transformation preserving the unit disk, we can send $w_0$ to either $0$, or $\infty$. 
The   operator $H$ whose projectivization is the latter M\"obius 
transformation conjugates $K$ to a triangular matrix. Then the corresponding variable change 
$\wt Y=HY$ satisfies the statement of the proposition.
\end{proof}

\subsection{A special class of  torus dynamical type systems with diagonal matrix \texorpdfstring{$K$}{K} and Heun equations}\label{subsec:spec-case}

Here we deal with the special class of torus dynamical type linear systems~\eqref{syst}, i.e., satisfying~\eqref{symrel} modulo $\mathbb{C} Id$, with diagonal matrix $K$. 
Applying scalar variable change $Y\mapsto z^\mu Y$ we can kill  its lower diagonal term: thus we get $K_{22}=0$. 
\begin{proposition} \label{ptdiag} Each torus dynamical type system~\eqref{syst} with $K=\diag(\nu,0)$ has form  
\begin{equation}Y'=\left(\frac{\left(\begin{matrix}\nu & 0\\ 0 & 0\end{matrix}\right)}z+
\frac{\left(\begin{matrix}\phi & b\\ -\bar b & \phi-\bar\nu-c\end{matrix}\right)}{z-\alpha}  + \frac{\left(\begin{matrix}\psi-\nu-c & 
-b\\ \bar b & \psi\end{matrix}\right)}{z-\alpha^{-1}}\right)Y, \ \ \ \alpha, c\in\mathbb{R},\ \ \alpha\neq0, \label{kdiag}\end{equation}
where the other parameters $\nu$, $b$, $\phi$, $\psi$ are complex numbers.
\end{proposition}
\begin{proof} The sum of  lower (upper) triangular elements of the matrices $R_1$ and $R_2$ 
should be zero, by the off-diagonal part of the second equation in 
~\eqref{symrel} and since $K$ is diagonal. Thus, 
 \begin{equation} b:=R_{1,12}=-R_{2,12}, \ \ R_{1,21}=-R_{2,21}.\label{tri0}\end{equation}
 On the other hand, the first equation in~\eqref{symrel} is 
 $$R_2=\overline R_1^{tt} \text{ modulo } \mathbb{C} Id, \ \ \text{ i.e., } \ R_2-\overline R_1^{tt}\in\mathbb{C} Id.$$
  By~\eqref{tri0}, it is equivalent to the  system of equations:
 \begin{equation}b=-R_{2,12}=\overline R_{2,21}=-\overline R_{1,21},\label{b=}\end{equation}
$$R_{2,11}-\overline R_{1,22}=R_{2,22}-\overline R_{1,11}.$$
  Setting 
 $$\phi:=R_{1,11}, \ \ \psi:=R_{2,22}$$
 we rewrite the latter equation as $R_{2,11}-\overline R_{1,22}=\psi-\bar\phi$, i.e., 
 \begin{equation}R_{1,22}=\phi+u, \ R_{2,11}=\psi+\bar u, \ u\in\mathbb{C}.\label{r21}\end{equation} 
The diagonal part of the second equation in~\eqref{symrel} states that the sum of the diagonal parts of the four matrices in question 
should be equal to a scalar matrix, i.e., should have equal diagonal terms. Taking into account~\eqref{r21} and that $K_{22}=0$,  $K_{11}=\nu$, this is equivalent to the equation   
$$\nu+\phi+\psi+\bar u=\bar\nu+\phi+\psi+u,$$
or equivalently, $\nu+\bar u$ is a real number. Therefore, 
$$u=-\bar\nu-c, \ c\in\mathbb{R}, \ R_{1,22}=\phi+u=\phi-\bar\nu-c, \ R_{2,11}=\psi-\nu-c.$$
Thus, in our case, being of torus dynamical type is equivalent to~\eqref{kdiag}.
\end{proof}

In what follows we deal with {\it normalized} systems~\eqref{kdiag}, in which 
\begin{equation}  b\in\mathbb{R}_{\geq0}, \ \ \a\in\mathbb{R}_+.\label{norms}\end{equation}
One can achieve conditions~\eqref{norms} by applying rotation by angle $\pi$ in the $z$-coordinate  and subsequent gauge transformation $(Y_1,Y_2)\mapsto(\varepsilon Y_1,Y_2)$, $|\varepsilon|=1$, rotating the coefficient $b$ to the real nonnegative semiaxis. This yields family of normalized systems 
\begin{equation} Y'=\left(\frac{\left(\begin{matrix}\nu & 0\\ 0 & 0\end{matrix}\right)}z+
\frac{\left(\begin{matrix}\phi & b\\ -b & \phi-\bar\nu-c\end{matrix}\right)}{z-\alpha}  + \frac{\left(\begin{matrix}\psi-\nu-c & 
-b\\  b & \psi\end{matrix}\right)}{z-\alpha^{-1}}\right)Y, \ \ \a\in\mathbb{R}_+, \   b\in\mathbb{R}_{\geq0}, \   c\in\mathbb{R}.  \label{kd}\end{equation}

\def\bam{\a^{-1}}

\begin{theorem} \label{heuntor} 1) For a system~\eqref{kd} with $\alpha\neq1$, $b\neq0$ the two following statements are equivalent:

(i) the system is equivalent to a Heun equation with singularities $0$, $\alpha$, $\a^{-1}$, $\infty$, i.e., 
for each solution $Y=(Y_1,  Y_2)$ of~\eqref{kd} its component $E=Y_2(z)$  
satisfies a Heun equation;

(ii)  One has  
\begin{equation}b^2=\phi(\bar\nu+c-\phi)=\psi(\nu+c-\psi)\in\mathbb{R}_{\geq0}; \ \ \text{ either } \ \psi=\bar\phi, \ \ \text{ or } \ \psi=\nu+c-\bar\phi.\label{psi}\end{equation}

2) If~\eqref{psi} holds, then 
 the corresponding Heun equation is 
\begin{equation} z(z-\bam)(z-\a)E''+(-\nu(z-\alpha)(z-\bam)+q z(z-\bam)+sz(z-\a))E'+(uz+d)E=0,
\label{heunt}\end{equation}
$$q=\bar\nu+c+1-2\phi, \ \ s=\nu+c+1-2\psi, \ u=(\bar\nu+c-\phi-\psi)(c+1-\phi-\psi), \ \ d=\nu((c+\bar\nu-\phi)\bam-\psi\a);$$
$$s=\begin{cases} \bar q, \ \text{ if } \ \psi=\bar\phi\\
s=2-\bar q, \ \text{ if } \psi=\nu+c-\bar\phi.\end{cases}$$
Thus, family of Heun equations~\eqref{heunt} depends on five real parameters: namely, 

-  a complex parameters  $\nu$;

- two real parameters $\alpha$ and $c$; 

- one complex parameter $\phi$ such that $\phi(\bar\nu+c-\phi)\in\mathbb{R}_+$;

- a double  choice for the parameter $\psi\in\{\bar\phi, \ \nu+c-\bar\phi\}$.
\end{theorem}
\begin{remark} In the crossed-out case $b=0$ the system is diagonal and the corresponding differential equation on $E$ is a Fuchsian equation with singularities 0, $\alpha$, $\alpha^{-1}$, $\infty$.
\end{remark}

\begin{proof}  The second equation, on $E=Y_2$ in~\eqref{kd} is 
$$E'=\frac{b(\bam-\a)}{(z-\a)(z-\bam)}Y_1+
\left(\frac{\phi-\bar\nu-c}{z-\a}+\frac\psi{z-\bam}\right)E,$$
which is equivalent to 
\begin{equation}Y_1=\frac1{ b(\bam-\a)}\left((z-\bam)(z-\a)E'-((\phi-\bar\nu-c)(z-\bam)+\psi(z-\a))E\right). 
\label{y1}\end{equation}
Differentiating~\eqref{y1} and multiplying it by $ b(\bam-\a)$ yields
%equating it with the first equation, on $Y_1$, in~\eqref{kd}), yields 
$$ b(\bam-\a)Y_1'=(z-\bam)(z-\a)E''$$
$$+((z-\a)+(z-\bam)-(\phi-\bar\nu-c)(z-\bam)-\psi(z-\a))E'-(\phi+\psi-\bar\nu-c)E$$
\begin{equation}=(z-\bam)(z-\a)E''-((\phi-\bar\nu-c-1)(z-\bam)+(\psi-1)(z-\a))E'-(\phi+\psi-\bar\nu-c)E. \label{y11}\end{equation}
Equating the latter right-hand side to 
the first equation, on $Y_1'$, in~\eqref{kd}, multiplied by $ b(\bam-\a)$, yields 
that it is equal to 
$$ b(\bam-\a)\left(\left(\frac{\nu}z+\frac{\phi}{z-\a}+\frac{\psi-\nu-c}{z-\bam}\right)Y_1+
b\left(\frac1{z-\a}-\frac1{z-\bam}\right)E\right).$$
Substituting here expression~\eqref{y1} for $Y_1$ yields 
$$\left(\frac{\nu}z+\frac{\phi}{z-\a}+\frac{\psi-\nu-c}{z-\bam}\right)\left((z-\bam)(z-\a)E'-((\phi-\bar\nu-c)(z-\bam)+\psi(z-\a))E\right)$$
$$+ b^2(\bam-\a)\left(\frac1{z-\a}-\frac1{z-\bam}\right)E=\left(\frac{\nu}z+(\phi+\nu)(z-\bam)+
(\psi-\nu-c)(z-\a)-\nu\a\right)E'$$
$$-\left(\frac{\nu}z+\frac{\phi}{z-\a}+\frac{\psi-\nu-c}{z-\bam}\right)((\phi-\bar\nu-c)(z-\bam)+\psi(z-\a))E$$
$$+ b^2(\bam-\a)\left(\frac1{z-\a}-\frac1{z-\bam}\right)E$$
$$=\left(\frac{\nu}z+(\phi+\nu)(z-\bam)+
(\psi-\nu-c)(z-\a)-\nu\a\right)E'
+\left(\frac{\nu}z((\phi-\bar\nu-c)\bam+\psi\a)\right.$$
\begin{equation}\left.-(\phi+\psi-c)(\phi+\psi-c-\bar\nu)+(\bam-\a)\left(\frac{ b^2-\phi(\bar\nu+c-\phi)}
{z-\a}-\frac{ b^2-\psi(\nu+c-\psi)}{z-\bam}\right)\right)E.\label{z-a}\end{equation}
Equating~\eqref{z-a} to~\eqref{y11} yields a linear second order differential equation in $E$. It is a Heun equation, 
if and only if the latter coefficients at $\frac1{z-\a}$, $\frac1{z-\bam}$ vanish. This is equivalent to system of equalities 
~\eqref{psi}. This proves Statement 1) of Theorem \ref{heuntor}.

In the latter case the Heun equation takes a form 
$$z(z-\a)(z-\bam)E''+(\text{degree 2 polynomial})\times E'+(uz+d)E=0.$$
The polynomial coefficient at $E'$ is found from~\eqref{y11} and~\eqref{z-a}: it is equal to 
$$(\bar\nu+c+1-2\phi-\nu)z(z-\bam)+(1+\nu+c-2\psi)z(z-\a)+\nu\alpha(z-\bam).$$
Writing it as $p(z-\bam)(z-\a)+qz(z-\bam)+sz(z-\a)$ with unknown coefficients $p$, $q$, $s$, we find them 
by substituting $z=0,\ \a, \ \bam$: this yields
$$p=-\nu, \ \ q=\bar\nu+c+1-2\phi-\nu+\nu=\bar\nu+c+1-2\phi, \ \ s=\nu+c+1-2\psi.$$
The polynomial coefficient $uz+d$ at $E$ is also found from~\eqref{y11} and~\eqref{z-a}:
$$u=(\bar\nu+c-\phi-\psi)(c+1-\phi-\psi), \ \ d=\nu((c+\bar\nu-\phi)\bam-\psi\a).$$
Theorem \ref{heuntor} is proved.
\end{proof}

\section{Confluent case}\label{sec:confluent-case}

Here we describe torus dynamical type linear systems with three singularities: Fuchsian singularities at 
$0$, $\infty$ and Poincar\'e rank 1 irregular singularity at $1$, 
with residue matrix at $0$ being diagonal $\diag(\nu,0)$. This family is a limit of family~\eqref{kdiag}, as $\a\to1$. We prove the next two theorems.

\begin{theorem} \label{tconfs} Each torus dynamical type linear system of the above type  can be 
normalized by gauge transformation $(Y_1,Y_2)\mapsto(\varepsilon Y_1,Y_2)$, $|\varepsilon|=1$, to a system of 
the form 
\begin{equation} Y'=\left(\frac{\left(\begin{matrix}\nu & 0\\ 0 & 0\end{matrix}\right)}z+
\frac{\left(\begin{matrix}a_2 & b\\ - b & a_2-g\end{matrix}\right)}{(z-1)^2}  + \frac{\left(\begin{matrix} 
a_1 & 0\\ 0 & a_1+\nu-\bar\nu\end{matrix}\right)}{z-1}\right)Y, \ \ \ b\in\mathbb{R}, \ \ g\in\mathbb{R},\label{y'c}\end{equation}
where the other parameters $\nu$, $a_1$, $a_2$ are complex.
\end{theorem}

Afterwards we discuss relation to confluent Heun equations. Recall that the 
standard family of confluent Heun equation is family of equations 
\begin{equation}\widehat E''_{ww}+\left(\frac\gamma{w}+\frac\delta{w-1}+\varepsilon\right)\widehat E'_w+\frac{\a w-q}{w(w-1)}\widehat E=0\label{heuncs}\end{equation}
which have Fuchsian singularities at $0$, $1$ and irregular singularity at $\infty$ of Poincar\'e rank 1.  We 
will deal with the following family of so-called {\it renormalized} 
 confluent Heun equations written in equivalent form, with 
Fuchsian singularities at $0$, $\infty$ and irregular singularity at $1$:
\begin{equation}z(z-1)^2E''+(pz(z-1)+qz+s)E'+(uz+d)E=0.\label{heunc}\end{equation}
Equations~\eqref{heunc} are obtained from equations~\eqref{heuncs} with $\varepsilon\neq0$ by changes $z=\frac w{w-1}$, $E=(z-1)^\lambda\widehat{E}$, $\lambda=-\frac{\alpha}{\varepsilon}$,  of variable and function. We prove the following theorem describing those systems (\ref{y'c}) that are equivalent to equations~\eqref{heunc} on the second components $E=Y_2(z)$ of their solutions $Y=(Y_1,Y_2)$. We describe equations (\ref{heunc}) that arise in this way.
\begin{theorem} \label{thc} 1) System~\eqref{y'c} with $b\neq0$ is equivalent to confluent Heun equation~\eqref{heunc} on 
 second component $E=Y_2(z)$ of its solution, if and only if its parameters satisfy the following system of 
 two equations:
 \begin{equation} b^2=a_2(g-a_2), \ \ \ a_2(\nu-\bar\nu)+a_1(2a_2-g)=0.\label{ech}\end{equation}
 
 2) In this case the corresponding coefficients of Heun equation~\eqref{heunc} are 
$$p=\bar\nu+2-2\nu-2a_1, \ \ q=\nu+g-2a_2, \ \ s=-\nu,$$
 \begin{equation} u=(a_1+\nu-1)(a_1+\nu-\bar\nu), \ \ d=\nu(a_2-g-a_1-\nu+\bar\nu).\label{coefl}
 \end{equation}
 \end{theorem}

\subsection{Torus dynamical type confluent systems. Proof of Theorem \ref{tconfs}}

\begin{proposition} \label{pcsym} A linear system 
\begin{equation}Y'=\left(\frac Az+\frac B{(z-1)^2}+\frac C{z-1}\right)Y\label{yconfl}\end{equation}
is of torus dynamical type, if and only if the following equalities hold modulo $\mathbb{C} Id$:
\begin{equation}A=-(\overline A^{tt}+\overline C^{tt}), \ \ B=-\overline B^{tt}, \ C=\overline C^{tt}.\label{abc}\end{equation}
\end{proposition}
\begin{proof} The proof is analogous to that of Proposition \ref{psym}. The variable changes 
 $Y\mapsto (\overline Y_2,\overline Y_1)$, $z\mapsto\bar z^{-1}$ replace the matrices $A$, $B$, $C$ of system 
~\eqref{yconfl} by $-(\overline A^{tt}+\overline C^{tt})$, $-\overline B^{tt}$, $\overline C^{tt}$ respectively, 
 as in the proof of Proposition \ref{psym}. This together with Remark \ref{remsym} implies the statement of Proposition \ref{pcsym}.
 \end{proof}
 
Consider a system~\eqref{yconfl} with $A=\diag(\nu,0)$. It is  of torus dynamical type, if and only if equalities~\eqref{abc} hold. The first equality in~\eqref{abc} is equivalent to the condition saying that $C$ is a diagonal matrix and the difference of its second and first diagonal elements is equal to $\nu-\bar\nu$. Then it automatically satisfies the third equality in~\eqref{abc} modulo $\mathbb{C} Id$. The second equality in 
~\eqref{abc} is equivalent to the condition that the matrix $B$ is as in~\eqref{y'c} after applying appropriate gauge transformation $(Y_1,Y_2)\mapsto(\varepsilon Y_1, Y_2)$, $|\varepsilon|=1$. Theorem \ref{tconfs} is proved.

\subsection{Corresponding confluent Heun equations. Proof of Theorem \ref{thc}}

The second differential equation, on $E=Y_2$, in~\eqref{y'c} is  
$$E'=-\frac{b}{(z-1)^2}Y_1+\left(\frac{a_2-g}{(z-1)^2}+\frac{a_1+\nu-\bar\nu}{z-1}\right)E,$$
which is equivalent to 
\begin{equation} Y_1=- b^{-1}\left((z-1)^2E'-(a_2-g+(z-1)(a_1+\nu-\bar\nu))E\right).\label{y1c}\end{equation}
Differentiating~\eqref{y1c} and equating it with the right-hand side of the first equation, on $Y_1$, in~\eqref{y'c}
yields
$$Y'_1= -b^{-1}\left((z-1)^2E''-(a_2-g+(z-1)(a_1+\nu-\bar\nu-2))E'-(a_1+\nu-\bar\nu)E\right)$$
$$=\left(\frac\nu{z}+\frac{a_2}{(z-1)^2}+\frac{a_1}{z-1}\right)Y_1+
\frac{b}{(z-1)^2}E.$$
Substituting formula~\eqref{y1c} for $Y_1$ to the latter expression, together with formulas  
$$\left(\frac\nu{z}+\frac{a_2}{(z-1)^2}+\frac{a_1}{z-1}\right)(z-1)^2=\frac\nu{z}+(a_1+\nu)(z-1)+a_2-\nu,$$
\begin{multline*}
\left(\frac\nu{z}+\frac{a_2}{(z-1)^2}+\frac{a_1}{z-1}\right)(a_2-g+(z-1)(a_1+\nu-\bar\nu))=
\frac{a_2(a_2-g)}{(z-1)^2}+\\+\frac{a_2(a_1+\nu-\bar\nu)+a_1(a_2-g)}{z-1}+\frac{\nu}z(a_2-g-a_1-\nu+\bar\nu)+(a_1+\nu)(a_1+\nu-\bar\nu),
\end{multline*}
which yields
$$(z-1)^2E''-\left((2a_1+2\nu-\bar\nu-2)(z-1)+2a_2-g-\nu+\frac\nu{z}\right)E'+\left(\frac{b^2-a_2(g-a_2)}{(z-1)^2}
\right.$$
\begin{equation}\left.+\frac{a_2(a_1+\nu-\bar\nu)+a_1(a_2-g)}{z-1}+\frac\nu{z}(a_2-g-a_1-\nu+\bar\nu)+
(a_1+\nu-1)(a_1+\nu-\bar\nu)\right)E=0.\label{heuncd}\end{equation}
Equation~\eqref{heuncd} is equivalent to~\eqref{heunc}, if and only if its coefficient at $E$ does not have a pole at $z=1$. The latter condition is equivalent to the system of equations~\eqref{ech}. In this case~\eqref{heuncd} yields formulas~\eqref{coefl} for the coefficients of the Heun equation. Theorem \ref{thc} is proved.

 \section{Dynamical systems on torus}\label{sec:dyn-tor-from-lin-sys}
 
\subsection[Case of four singularities]{Case of four singularities $0$, $\a>0$, $\a^{-1}$, $\infty$; $\alpha\neq1$}\label{subsec:four-sing-proof}

 Linear system~\eqref{kd} on a vector function $Y=(Y_1(z),Y_2(z))$ together with the tautological projection 
 $$\mathbb{C}^2_{Y_1,Y_2}\setminus\{0\}\to\mathbb{CP}^1[Y_1:Y_2]=\overline{\mathbb{C}}_{\Phi}, \ \ \ \Phi:=\frac{Y_2}{Y_1},$$
 induce the following Riccati equation on the function $\Phi(z)$, the projectivization of~\eqref{kd}: 
 \begin{equation}\Phi'=-\frac{b(\difa)}{\zaal}(1+\Phi^2)-\left(\frac{\bar\nu+c}{z-\a}-\frac{\nu+c}{z-\bam}+
 \frac{\nu}z\right)\Phi, \ \ \a>0, \ b\geq0, \ c\in\mathbb{R}.\label{ric1}\end{equation}
 \begin{proposition} \label{prodyn1} The variable change 
 $$\Phi=e^{i\theta}, \ \ \ z=e^{i\tau}$$
 transforms the restriction of Riccati equation~\eqref{ric1} to the  torus $\mathbb{T}^2=\mathbb{R}^2_{\theta,\tau}\slash 2\pi\mathbb{Z}^2$ 
 to the following differential equation on torus:
 \begin{equation}\frac{d\theta}{d\tau}=\frac{\cos\theta+B+A\sin\tau}{\omega(1-\delta\cos\tau)}+D,
 \label{dyn}\end{equation}
$$\omega=\frac{\a+\bam}{2b(\difa)}, \ \ \  \ D=-\re\nu, \ \ \ \ \delta=\frac2{\a+\bam}\in(0,1),$$
 \begin{equation} 
 B=\frac{c+\re\nu}{2b}, \ \ \ \ \ \ \  A=\frac{\im\nu}{b(\difa)}.\label{coefd}\end{equation}
 \end{proposition}
 \begin{proof} One has ${d\Phi}/{d\tau}=iz({d\Phi}/{dz})=i\Phi({d\theta}/{d\tau})$. Therefore,
\begin{equation}\frac{d\theta}{d\tau}=\frac{z}{\Phi}\frac{d\Phi}{dz}=-\frac{b(\difa)z}{\zaal}(\Phi+\Phi^{-1})-
 \left(\frac{(\bar\nu+c)z}{z-\a}-\frac{(\nu+c)z}{z-\bam}+
\nu\right).\label{exth}\end{equation}
The first ratio (taken with sign ``$-$'') in the  right-hand side in~\eqref{exth} is equal to 
$$\frac{b(\bam-\a)z}{\zaal}=\frac{b(\bam-\a)e^{i\tau}}{(e^{i\tau}-\a)(e^{i\tau}-\bam)}=
\frac{b(\difa)}{(\a+\bam)-e^{i\tau}-e^{-i\tau}}=\frac{b(\difa)}{\a+\bam-2\cos\tau}.$$
 The expression in big brackets in the same right-hand side is equal to 
 $$\frac{(c+\re\nu)(\difa)z}{\zaal}+i\im\nu\left(1-\frac z{z-\a}-\frac z{z-\bam}\right)+\re\nu$$
 $$=
 -\frac{(c+\re\nu)(\difa)}{\zaaa}+i\im\nu\frac{z(z^{-1}-z)}{\zaa}+\re\nu$$
 $$=-\frac{(c+\re\nu)(\difa)}{\zaaa}+i\im\nu\frac{2i\sin\tau}{\zaaa}+\re\nu.$$
 Substituting the above equalities and $\Phi^{-1}+\Phi=2\cos\theta$ (for $|\Phi|=1$)  to~\eqref{exth} yields
 $$\frac{d\theta}{d\tau}=\frac{(\difa)(2b\cos\theta+(c+\re\nu))+2\im\nu\sin\tau}{\a+\bam-2\cos\tau}-\re\nu.$$
 This proves~\eqref{dyn} with parameters given by~\eqref{coefd}.
 \end{proof}

\subsection[Confluent case: three singularities]{Confluent case: three singularities $0$, $1$, $\infty$}\label{subsec:three-sings}
Linear system~\eqref{y'c} on a vector function $Y=(Y_1(z),Y_2(z))$ together with the tautological projection 
 $$\mathbb{C}^2_{Y_1,Y_2}\setminus\{0\}\to\mathbb{CP}^1[Y_1:Y_2]=\overline{\mathbb{C}}_{\Phi}, \ \ \ \Phi:=\frac{Y_2}{Y_1},$$
 induce the following Riccati equation on the function $\Phi(z)$, the projectivization of~\eqref{y'c}: 
 \begin{equation}\Phi'=-\frac{b}{(z-1)^2}(1+\Phi^2)+\left(-\frac{\nu}{z}-\frac{g}{(z-1)^2}+\frac{\nu-\bar\nu}{z-1}\right)\Phi, \ \  \ b, g\in\mathbb R.\label{ric2}\end{equation}

\begin{proposition} \label{prodyn2} The variable change $\Phi=e^{i\theta}$, $z=e^{i\tau}$ 
 transforms the restriction of Riccati equation~\eqref{ric2} to the  torus $\mathbb{T}^2=\mathbb{R}^2_{\theta,\tau}\slash2\pi\mathbb{Z}^2$ 
 to the following differential equation on torus:
 \begin{equation}\frac{d\theta}{d\tau}=\frac{\cos\theta+B+A\sin\tau}{\omega(1-\cos\tau)}+D,
 \label{dyn2}\end{equation}
 \begin{equation} \omega=\frac1b, \ \ B=\frac{g}{2b}, \ \ A=\frac{\im\nu}{b}, \ \ D=-\re\nu.\label{confcs}
 \end{equation}
 \end{proposition}
 \begin{proof} As in the proof of Proposition \ref{prodyn1}, the corresponding  
 differential equation on the torus states that the derivative ${d\theta}/{d\tau}$ is equal to the right-hand side of~\eqref{dyn2} multiplied by $z/{\Phi}$. This together with formulas (here $|\Phi|=|z|=1$) 
 $$\Phi^{-1}+\Phi=2\cos\theta, \ \ \  \frac z{(z-1)^2}=\frac1{e^{i\tau}+e^{-i\tau}-2}=-\frac1{2(1-\cos\tau)},$$
 $$-\nu+\frac{(\nu-\bar\nu)z}{z-1}=-\re\nu+i\im\nu\left(\frac{2z}{z-1}-1\right),$$ 
 $$\frac{2z}{z-1}-1=\frac{z+1}{z-1}=\frac{(z+1)(\bar z-1)}{|z-1|^2}=\frac{e^{-i\tau}-e^{i\tau}}{2(1-\cos\tau)}=-i\frac{\sin\tau}{1-\cos\tau}$$
 yields
 $$\frac{d\theta}{d\tau}=\frac{b\cos\theta}{1-\cos\tau}+\left(-\re\nu+\frac{g}{2(1-\cos\tau)}+\im\nu\frac{\sin\tau}{1-\cos\tau}\right),$$
 which implies~\eqref{dyn2} with coefficients as in~\eqref{confcs}. 
 \end{proof}

\subsection{Family dRSJ and constriction breaking}\label{subsec:dRSJ-family}

Consider a differential equation on torus $\mathbb T^2=\mathbb R^2_{\theta,\tau}\slash2\pi\mathbb Z^2$ 
of the type $\frac{d\theta}{d\tau}=f(\theta,\tau)$; $f$ is $2\pi$-periodic in $\theta$ and in $\tau$. For example, 
an equation (\ref{dyn3}). Recall that its {\it Poincar\'e first return map} is the diffeomorphism of the transversal circle $S^1=S^1_{\theta}\times\{0\}$ given by the 
time $2\pi$ flow map. This is the map that sends an initial condition in $S^1$ to the point of the first return  to $S^1$ of the corresponding solution in positive time. 

\begin{theorem} \label{montr} Consider the 5-parameter dynamical system family~\eqref{dyn3}. 

1) Let  $\delta\in[0,1)$ and $A\in\mathbb{R}$ be both non-zero. Then 
for arbitrary choice of the other parameters $\omega\in\mathbb R\setminus\{0\}$, $B,D\in\mathbb{R}$ 
the Poincar\'e map of system~\eqref{dyn3} is not the identity. 

2) If a system~\eqref{dyn3} with $A=0$ has identity Poincar\'e map, then 
\begin{equation}B^2-1=\omega^2(1-\delta^2)(D-n)^2, \ \ n\in\mathbb{Z}.\label{montriv}\end{equation}
\end{theorem}
\begin{proof} Let us prove Statement 1). Suppose the contrary: the Poincar\'e map is identity for some parameter values with $\delta\neq 0$, $A\neq0$. Then the corresponding Riccati equation~\eqref{ric3} has trivial monodromy along the counterclockwise unit circle in the Riemann sphere $\overline{\mathbb{C}}_z$. Equivalently, the corresponding linear system~\eqref{kd3} has scalar monodromy $M=\lambda Id$, $\lambda\in\mathbb{C}^*$. It has two finite non-zero singularities: the roots $\alpha^{\pm1}$ of the quadratic polynomial $2z-\delta(z^2+1)$. We label them  so that $\alpha\in(0,1)$. Then the monodromy operator $M$ is the product of the monodromy operators $M_0$, $M_\alpha$ around the singular points $0$ and $\alpha$ lying in the unit disk. Projective triviality of $M$ implies coincidence of eigenvectors of the operators  $M_0$ and $M_\alpha$ and the fact that the ratio of eigenvalues of $M_0$ is equal to that of $M_\alpha$ up to taking inverse. 

Let $\lambda_0$, $\lambda_{\alpha}$ denote the differences of the residue matrix eigenvalues at $0$ and at $\alpha$ respectively. The above eigenvalue ratio relation is equivalent to  the relation 
$$\lambda_0=\pm\lambda_{\alpha}-m, \ \ \ m\in\mathbb{Z},$$
which implies 
\begin{equation}(\lambda_0+m)^2=\lambda_\alpha^2, \ \ \ m\in\mathbb{Z}.\label{laoa}\end{equation}
The direct calculation of residue eigenvalues in~\eqref{kd3} yields
$$\lambda_0=\pm\nu, \ \ \lambda_{\alpha}=\pm\sqrt{(\bar\nu+c)^2-4b^2}.$$
This together with~\eqref{laoa} yields $(\nu+n)^2=(\bar\nu+c)^2-4b^2$, $n=\pm m\in\mathbb{Z}$, thus, 
\begin{equation}(2\re\nu+c+n)(-2i\im\nu+c-n)=4b^2.\label{lanu}\end{equation}
Recall that $b$ is real and non-zero. This together with~\eqref{lanu} implies that 
$\im\nu=0$. But then $A=\im\nu/[b(\a-\a^{-1})]=0$. The contradiction thus obtained 
proves Statement 1).

Let us prove Statement 2). Let the monodromy be trivial, $\delta\in(0,1)$ and $A=0$; thus, 
$\im\nu=0$. 
Then relation~\eqref{lanu} yields
$$(c+\re\nu)^2-(n+\re\nu)^2=4b^2.$$
Substituting there formulas $c+\re\nu=2bB$, $\re\nu=-D$, see~\eqref{coefd3}, yields
\begin{equation}4b^2(B^2-1)=(D-n)^2, \ \ n\in\mathbb{Z}.\label{bbd}\end{equation}
One has
\begin{equation}b=\frac1{\delta \omega(\alpha-\alpha^{-1})}, \ \ \ (\alpha-\alpha^{-1})^2=\frac{4(1-\delta^2)}{\delta^2},\label{ba}\end{equation} 
respectively by~\eqref{coefd3} and the discriminant formula for the quadratic polynomial $$z(1-\delta(z+z^{-1})/2)=-\frac{\delta}2(z-\alpha)(z-\alpha^{-1}).$$
Substituting formulas~\eqref{ba} to~\eqref{bbd} yields $(B^2-1)/(\omega^2(1-\delta^2))=(D-n)^2$, which implies~\eqref{montriv}. Theorem \ref{montr} is proved.
\end{proof}

Consider family~\eqref{dyn3}$_{\omega, \delta,  D}$ of dynamical systems~\eqref{dyn3} on torus, in which we fix the three parameters $\omega>0$, $\delta\in[0,1)$ and $D$; for example, $D=0$. 
This is a family of dynamical systems that depend on two parameters $(B,A)$. 
If $D=0$, then as $\delta\to0$, the system~\eqref{dyn3}$_{\omega, \delta,  0}$ degenerates to the RSJ model. Thus, the family~\eqref{dyn3}$_{\omega, \delta,  0}$ can be considered as a deformation of the RSJ model~\eqref{dyn3}$_{\omega, 0,  0}$: 
\begin{equation}\frac{d\theta}{d\tau}=\frac1{\omega}(\cos\theta+B+A\sin\tau).\label{jos0}
\end{equation} 
Recall that for every $r\in\mathbb{Z}$ the $r$-th phase-lock area $L_r$ in $\mathbb{R}^2_{B,A}$ of the non-perturbed system, i.e.,~\eqref{jos0}, is a garland of infinitely many connected components of its interior that tend to infinity in the vertical direction.  Each of the two neighbor components is separated by one point. Those of the separation points that do not lie in the abscissa axis, i.e., for which $A\neq0$, are called {\it constrictions.} It is known that all the constrictions in $L_r$ lie in the vertical line $\{ B=\omega r\}$, see~\cite{bibilo2022families}.
 
\begin{corollary} \label{cmontr} Let $D=0$, $\omega\in\mathbb R\setminus\{0\}$ be fixed. As we perturb system~\eqref{jos0} to system~\eqref{dyn3}$_{\omega, \delta,  D}$ with small fixed $\delta\neq0$ and arbitrary small $D\in\mathbb R$, then all the constrictions break down. Thus, the intersections of the interiors of the phase-lock areas with upper and  lower half-planes $\{\pm A>0\}$ become connected. See Fig.~\ref{fig:phase-lock-areas-AB-plane} for $\omega=1$.
\end{corollary}

\begin{figure}[h!]
\centering
\includegraphics[width=0.8\linewidth]{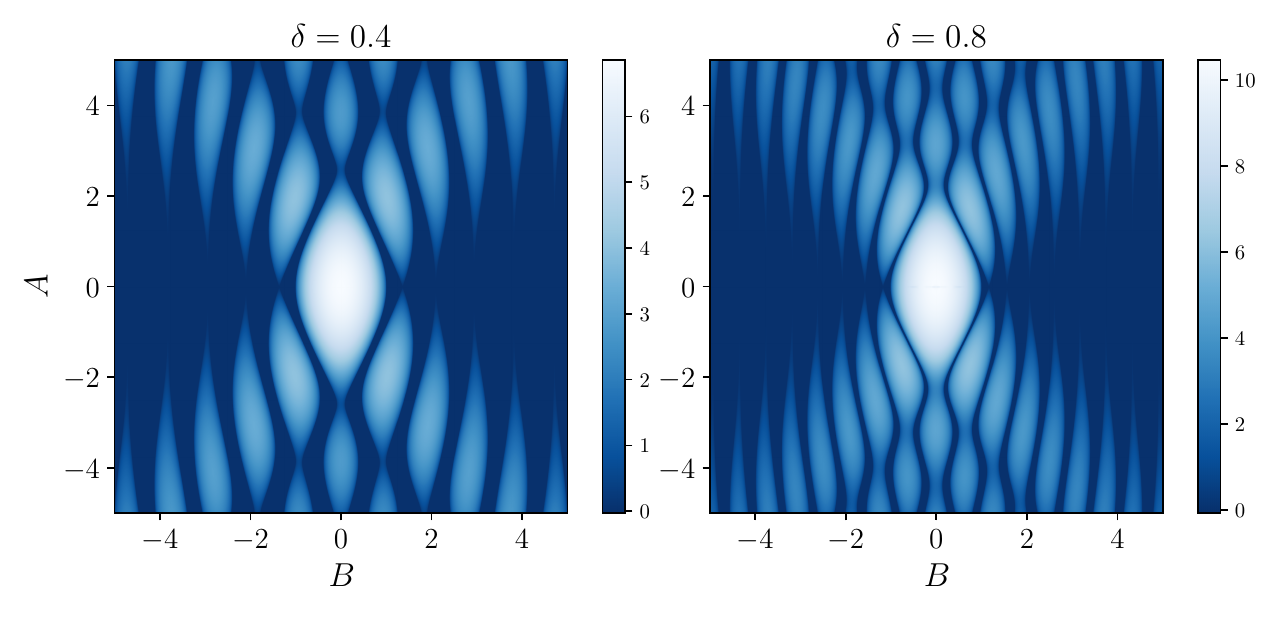}
\caption{Phase-lock areas of system~\eqref{dyn3} in $(B,A)$-plane without constrictions.  Here $\omega=1$, $D=0$. Color bar represents the values of the largest Lyapunov exponent.}
\label{fig:phase-lock-areas-AB-plane}
\end{figure}

\begin{figure}[h!]
\centering
\includegraphics[width=0.8\linewidth]{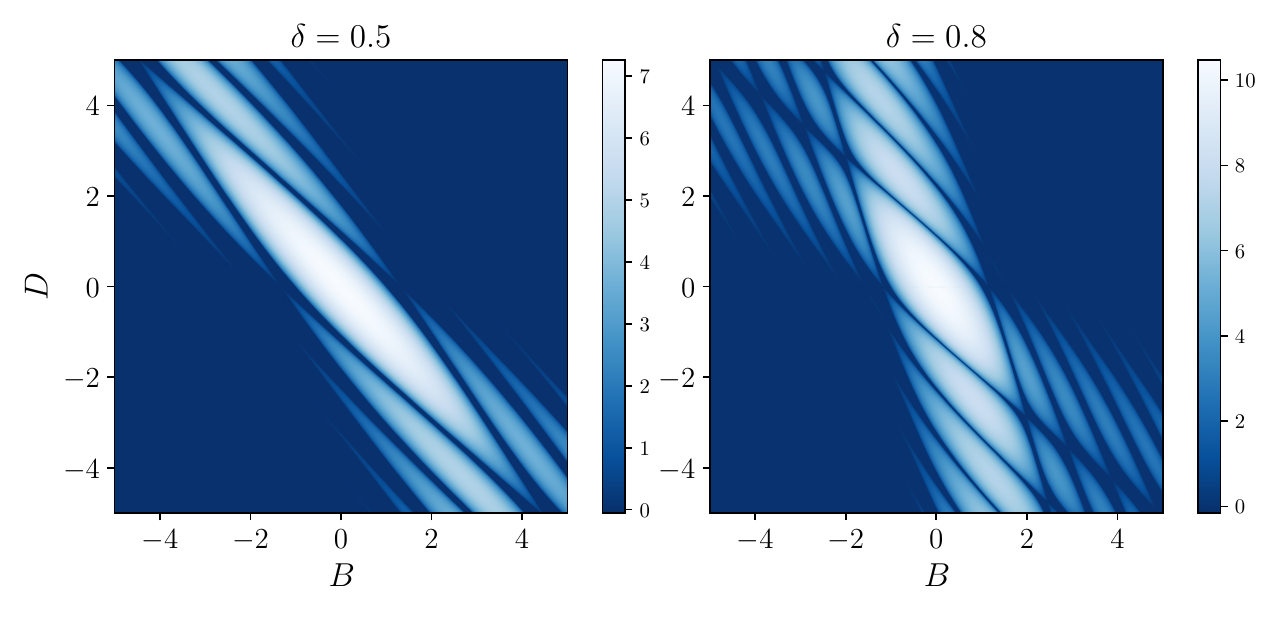}
\caption{Phase-lock areas of system~\eqref{dyn3} with $A=0$, $\omega=1$ in the $(B,D)$-plane. Color bar represents the values of the largest Lyapunov exponent.}
\label{fig:phase-lock-areas-BD-plane}
\end{figure}

\begin{proof} Consider an arbitrary family of dynamical systems depending on parameters $(B,A)$ and depending strictly monotonously on $B$. Then each constriction of a phase-lock area (if any) always corresponds to a dynamical system with trivial Poincar\'e map, see \cite[proposition 2.2]{glutsyuk2014adjacency}. Though this proposition was stated in a special case, it remains valid in full generality together with its proof. This together with Theorem \ref{montr} implies that if $\delta\neq0$, then there are no constrictions with $A\neq0$.
\end{proof}

Consider the phase-lock area portrait of family (\ref{dyn3}) in the $(B,D)$-plane, see Fig. 3 for $\omega=1$. Despite the fact that the constrictions are broken down for $\delta>0$, the growth line of the phase-lock areas remains. This is  the line $\{ A=D=0\}$. The location of the growth points on this line is determined by the condition~\eqref{montriv}. Let us also demonstrate that the location of the growth points can be computed in another way.

\begin{theorem} \label{growd3}
Consider the family of systems~\eqref{dyn3} with $A=D=0$ $B>1$. Then the monodromy $M$ of the corresponding linear system is trivial if and only if 
$$B^2=\omega^2(1-\delta^2)n^2+1, \ \ \ n\in\mathbb{Z}.$$
\end{theorem}

\begin{proof}

Consider the system with $A=0$ and $D=0$,
\begin{equation}\label{eq:zero-AD-tor-sys}
    \frac{d\theta}{d\tau} = \frac{B+\cos\theta}{\omega(1-\delta\cos\tau)}.
\end{equation}
Let $z=e^{i\tau}$, $\Phi=e^{i\theta}$, which gives us
\begin{equation}\label{eq:Riccati-sys}
    \frac{d\Phi}{dz}=-\frac{1+2B\Phi+\Phi^2}{(\delta(1+z^2)-2z)\omega}.
\end{equation}
It is the Riccati equation obtained by  projectivization of the linear system 
\begin{equation}\label{eq:lin-sys}
    \frac{d}{dz}
    \begin{pmatrix}
        u \\ v
    \end{pmatrix}=
    \frac{1}{(\delta(1+z^2)-2z)\omega}
    \begin{pmatrix}
        2B & 1 \\ -1 & 0
    \end{pmatrix}
    \begin{pmatrix}
        u \\ v
    \end{pmatrix}:
\end{equation}
$\Phi(z)$ is a solution of~\eqref{eq:Riccati-sys}, if and only if $\Phi(z)=\frac{v(z)}{u(z)}$, where $(u(z),v(z))$ is a solution of system~\eqref{eq:lin-sys}. 
We are interested in the monodromy of this linear system, computed on the counterclockwise unit circle contour. To compute it, we need to find poles of the denominator in the right hand side of eq.~\eqref{eq:Riccati-sys}. They are given by
\begin{equation}
    z_{\pm} = \frac{1\pm\sqrt{1-\delta^2}}{\delta}.
\end{equation}
Only the point $z_{-}$ belongs to the unit disc, the integration contour interior. Thus, the point $z_+$ is not relevant for our goals. Denote the matrix in the right-hand side of~\eqref{eq:lin-sys} by $Q$,
\begin{equation}
    Q=
    \begin{pmatrix}
        2B & 1 \\ -1 & 0
    \end{pmatrix}.
\end{equation}
This matrix has the following eigenvalues,
\begin{equation}
    \lambda_{1,2}=B\mp\sqrt{B^2-1}.
\end{equation}
Thus, it is diagonalizable: $\lambda_1\neq \lambda_2$, since $B>1$. 
Since the point $z_{-}$ is a Fuchsian singularity, the monodromy matrix is conjugated to 
\begin{equation}
    M = \exp\left\{2\pi i\mathop{\text{Res}}_{z=z_{-}}\frac{\diag(\lambda_1,\,\lambda_2)}{\omega(\delta(1+z^2)-2z)}\right\}=
    \begin{pmatrix}e^{2\pi i\mu_1} & \\ 0 & e^{2\pi i\mu_2}\end{pmatrix},
\end{equation}
where the quantities $\mu_1$ and $\mu_2$ are defined by
\begin{equation}
    \mu_1=\frac{-B+\sqrt{B^2-1}}{2\omega\sqrt{1-\delta^2}},\quad \mu_2=\frac{-B-\sqrt{B^2-1}}{2\omega\sqrt{1-\delta^2}}.
\end{equation}
The monodromy $M$ is scalar, if and only if the ratio of eigenvalues is equal to unity,
\begin{equation*}
    \frac{e^{2\pi i\mu_1}}{e^{2\pi i\mu_2}}=e^{2\pi i(\mu_1-\mu_2)}=1.
\end{equation*}
This is equivalent to the equality
\begin{equation}
   \mu_1-\mu_2= \frac{\sqrt{B^2-1}}{\omega\sqrt{1-\delta^2}}=n,\quad n\in\mathbb{Z}.
\end{equation}
Rearranging this expression, we find that the monodromy $M$ is trivial,  if and only if the parameters obey
\begin{equation}
    B^2=\omega^2n^2\left(1-\delta^2\right)+1,
\end{equation}
which in fact coincides with eq.~\eqref{montriv} with $D=0$.
\end{proof}

\begin{proposition} \label{crot}
The rotation number $\rho$ of~\eqref{dyn3} for $A=D=0$, $\omega>0$ and  $B>1$ is given by
\begin{equation}
    \rho = \frac{1}{2\pi i}\ln\frac{e^{2\pi i\mu_1}}{e^{2\pi i\mu_2}}=\mu_1-\mu_2 = \frac{\sqrt{B^2-1}}{\omega\sqrt{1-\delta^2}}.\label{m1m2}
\end{equation}
\end{proposition}
\begin{proof} Formula~\eqref{m1m2} holds modulo $\mathbb{Z}$ and modulo sign.  Indeed,  in our case the Poincar\'e map is an elliptic M\"obius transformation. Its rotation number, which is a number modulo $\mathbb Z$, is equal to $\frac1{2\pi}$ times the argument of its  derivative at its fixed point up to sign: in our case it is the right-hand side in~\eqref{m1m2}.  For $|B|\leq1$ the dynamical system equation has constant solution $\theta\equiv\arccos B$, and hence, $\rho=0$. The rotation number is clearly continuous and non-decreasing as a function in $B$, as is the dynamical system. This together with the above argument implies~\eqref{m1m2}.
\end{proof}

\section{Conclusion}

In this work, we have introduced two families of the dynamical systems on the 2D torus that are directly related to the general (non-confluent) Heun equation (GHE) and the confluent Heun equation (CHE). This is done by considering appropriate class of linear systems of torus dynamical type that are equivalent to Heun equations on the second component of vector solution. The corresponding dynamical systems are obtained by taking projectivization of linear systems, i.e., Riccati equations,  and their subsequent restriction to the product of unit circles in space and time variables. The family related to GHE, the so-called deformed RSJ model (dRSJ), demonstrates an interesting property in terms of phase lock areas: in this family constrictions are broken, whereas the growth points still exist.

From a physical point of view, it will be interesting and reasonable to find a realistic physical system, where the dRSJ describes dynamics (or may be effective dynamics). As one can see, the dRSJ model implies that the critical current becomes a function of time. We assume that one of the following systems can be used as a starting point to realize the desired dRSJ dynamics: SQUIDs, where the critical current is sensitive to the flux and external magnetic field; systems where Cooper pair breaking can occur, which implies that the superconducting gap can change in time; Josephson junctions with additional ferromagnetic barriers that can affect the critical current via spin dynamics.

Another possible interpretation of the obtained systems and results is related to the discussion of quasi-normal modes in quantum mechanical problems of different nature. We postpone these questions for future work. From a mathematical point of view, some of the problems related to these two families of dynamical systems remain open. It will be interesting to resolve the mentioned open problems and obtain a general picture that relates different special cases of the Heun equation (i.e. confluent, double-confluent, etc.) to a system on the 2D torus (if possible) and to understand in detail the structure of their phase lock areas. 

\section*{Acknowledgements}

We are grateful to A.S. Gorsky for a statement of the problem to find dynamical systems and phase locking phenomena for general Heun equations and for helpful discussions. The research of A. Aleksandrov is supported by Basic Research Program at HSE University. The research of A. Glutsyuk is supported by the MSHE project No. FSMG-2024-0048 and by grant No. 24-7-1-15-1 of the Theoretical Physics and Mathematics Advancement Foundation ``BASIS''.

\bibliographystyle{abbrvnat}
\bibliography{ref}

\begin{thebibliography}{43}
\providecommand{\natexlab}[1]{#1}
\providecommand{\url}[1]{\texttt{#1}}
\expandafter\ifx\csname urlstyle\endcsname\relax
  \providecommand{\doi}[1]{doi: #1}\else
  \providecommand{\doi}{doi: \begingroup \urlstyle{rm}\Url}\fi

\bibitem[Alexandrov et~al.(2025)Alexandrov, Glutsyuk, and Gorsky]{alexandrov2025}
A.~Alexandrov, A.~Glutsyuk, and A.~Gorsky.
\newblock Phase-locking in dynamical systems and quantum mechanics, 2025.
\newblock URL \url{https://arxiv.org/abs/2504.20181}.

\bibitem[Aminov et~al.(2022)Aminov, Grassi, and Hatsuda]{aminov2022black}
G.~Aminov, A.~Grassi, and Y.~Hatsuda.
\newblock Black hole quasinormal modes and {S}eiberg-{W}itten theory.
\newblock \emph{Annales Henri Poincare}, 23\penalty0 (6):\penalty0 1951--1977, 2022.
\newblock \doi{10.1007/s00023-021-01137-x}.
\newblock URL \url{https://doi.org/10.1007/s00023-021-01137-x}.

\bibitem[Arnold(2012)]{arn}
V.~Arnold.
\newblock \emph{Geometrical methods in the theory of ordinary differential equations}, volume 250.
\newblock Springer Science \& Business Media, 2012.

\bibitem[Bibilo and Glutsyuk(2022)]{bibilo2022families}
Y.~Bibilo and A.~Glutsyuk.
\newblock On families of constrictions in model of overdamped {J}osephson junction and {P}ainlev{\'e} 3 equation.
\newblock \emph{Nonlinearity}, 35\penalty0 (10):\penalty0 5427, 2022.
\newblock \doi{10.1088/1361-6544/ac8aee}.
\newblock URL \url{https://doi.org/10.1088/1361-6544/ac8aee}.

\bibitem[Buchstaber and Glutsyuk(2016)]{bg}
V.~Buchstaber and A.~Glutsyuk.
\newblock On determinants of modified {B}essel functions and entire solutions of double confluent {H}eun equations.
\newblock \emph{Nonlinearity}, 29\penalty0 (12):\penalty0 3857, 2016.
\newblock \doi{10.1088/0951-7715/29/12/3857}.
\newblock URL \url{https://doi.org/10.1088/0951-7715/29/12/3857}.

\bibitem[Buchstaber and Glutsyuk(2017)]{buchstaber2017monodromy}
V.~Buchstaber and A.~Glutsyuk.
\newblock On monodromy eigenfunctions of {H}eun equations and boundaries of phase-lock areas in a model of overdamped {J}osephson effect.
\newblock \emph{Proceedings of the Steklov Institute of Mathematics}, 297:\penalty0 50--89, 2017.
\newblock \doi{10.1134/S0081543817040046}.
\newblock URL \url{https://doi.org/10.1134/S0081543817040046}.

\bibitem[Buchstaber and Tertychniy(2013{\natexlab{a}})]{bt0}
V.~Buchstaber and S.~Tertychniy.
\newblock Explicit solution family for the equation of the resistively shunted {J}osephson junction model.
\newblock \emph{Theoretical and Mathematical Physics}, 176\penalty0 (2):\penalty0 965--986, 2013{\natexlab{a}}.
\newblock \doi{10.1007/s11232-013-0085-2}.
\newblock URL \url{https://doi.org/10.1007/s11232-013-0085-2}.

\bibitem[Buchstaber and Tertychniy(2013{\natexlab{b}})]{buchstaber2013explicit}
V.~Buchstaber and S.~Tertychniy.
\newblock Explicit solution family for the equation of the resistively shunted {J}osephson junction model.
\newblock \emph{Theoretical and Mathematical Physics}, 176\penalty0 (2):\penalty0 965--986, 2013{\natexlab{b}}.
\newblock \doi{10.1007/s11232-013-0085-2}.
\newblock URL \url{https://doi.org/10.1007/s11232-013-0085-2}.

\bibitem[Buchstaber and Tertychnyi(2015)]{buchstaber2015holomorphic}
V.~Buchstaber and S.~Tertychnyi.
\newblock Holomorphic solutions of the double confluent {H}eun equation associated with the {RSJ} model of the {J}osephson junction.
\newblock \emph{Theoretical and Mathematical Physics}, 182:\penalty0 329--355, 2015.
\newblock \doi{10.1007/s11232-015-0267-1}.
\newblock URL \url{https://doi.org/10.1007/s11232-015-0267-1}.

\bibitem[Buchstaber et~al.(2004)Buchstaber, Karpov, and Tertychnyi]{bkt1}
V.~Buchstaber, O.~Karpov, and S.~Tertychnyi.
\newblock On properties of the differential equation describing the dynamics of an overdamped {J}osephson junction.
\newblock \emph{Russian Mathematical Surveys}, 59\penalty0 (2):\penalty0 377, 2004.
\newblock \doi{10.1070/RM2004v059n02ABEH000725}.
\newblock URL \url{https://doi.org/10.1070/RM2004v059n02ABEH000725}.

\bibitem[Buchstaber et~al.(2010)Buchstaber, Karpov, and Tertychniy]{buchstaber2010rotation}
V.~Buchstaber, O.~Karpov, and S.~Tertychniy.
\newblock Rotation number quantization effect.
\newblock \emph{Theoretical and Mathematical Physics}, 162:\penalty0 211--221, 2010.
\newblock \doi{10.1007/s11232-010-0016-4}.
\newblock URL \url{https://doi.org/10.1007/s11232-010-0016-4}.

\bibitem[Dinsmore et~al.(2008)Dinsmore, Bae, and Bezryadin]{dinsmore2008}
R.~C. Dinsmore, M.~Bae, and A.~Bezryadin.
\newblock Fractional order {S}hapiro steps in superconducting nanowires.
\newblock \emph{Applied physics letters}, 93\penalty0 (19), 2008.
\newblock \doi{10.1063/1.3012360}.
\newblock URL \url{https://doi.org/10.1063/1.3012360}.

\bibitem[Figueiredo(2024)]{figueiredo2024}
B.~Figueiredo.
\newblock Schr{\"o}dinger equation as a confluent {H}eun equation.
\newblock \emph{Physica Scripta}, 99\penalty0 (5):\penalty0 055211, 2024.
\newblock \doi{10.1088/1402-4896/ad3510}.
\newblock URL \url{https://doi.org/10.1088/1402-4896/ad3510}.

\bibitem[Fiziev and Staicova(2011)]{fiziev2011application}
P.~Fiziev and D.~Staicova.
\newblock Application of the confluent {H}eun functions for finding the quasinormal modes of nonrotating black holes.
\newblock \emph{Physical Review D}, 84\penalty0 (12):\penalty0 127502, 2011.
\newblock \doi{10.1103/PhysRevD.84.127502}.
\newblock URL \url{https://doi.org/10.1103/PhysRevD.84.127502}.

\bibitem[Foote(1998)]{Foote}
R.~Foote.
\newblock Geometry of the {P}rytz planimeter.
\newblock \emph{Reports on mathematical physics}, 42\penalty0 (1-2):\penalty0 249--271, 1998.
\newblock \doi{10.1016/S0034-4877(98)80013-X}.
\newblock URL \url{https://doi.org/10.1016/S0034-4877(98)80013-X}.

\bibitem[Foote et~al.(2013)Foote, Levi, and Tabachnikov]{foott}
R.~Foote, M.~Levi, and S.~Tabachnikov.
\newblock Tractrices, bicycle tire tracks, hatchet planimeters, and a 100-year-old conjecture.
\newblock \emph{The American Mathematical Monthly}, 120\penalty0 (3):\penalty0 199--216, 2013.
\newblock \doi{10.4169/amer.math.monthly.120.03.199}.
\newblock URL \url{https://doi.org/10.4169/amer.math.monthly.120.03.199}.

\bibitem[Glutsyuk(2019)]{glutsyuk2019constrictions}
A.~Glutsyuk.
\newblock On constrictions of phase-lock areas in model of overdamped {J}osephson effect and transition matrix of the double-confluent {H}eun equation.
\newblock \emph{Journal of Dynamical and Control Systems}, 25\penalty0 (3):\penalty0 323--349, 2019.
\newblock \doi{10.1007/s10883-018-9411-1}.
\newblock URL \url{https://doi.org/10.1007/s10883-018-9411-1}.

\bibitem[Glutsyuk(2023)]{glutsyuk2023-germs}
A.~Glutsyuk.
\newblock On germs of constriction curves in model of overdamped josephson junction, dynamical isomonodromic foliation and painlev\'e 3 equation.
\newblock \emph{Moscow Mathematical Journal}, 23\penalty0 (4):\penalty0 479--513, 2023.

\bibitem[Glutsyuk and Rybnikov(2016)]{glutryb}
A.~Glutsyuk and L.~Rybnikov.
\newblock On families of differential equations on two-torus with all phase-lock areas.
\newblock \emph{Nonlinearity}, 30\penalty0 (1):\penalty0 61, 2016.
\newblock \doi{10.1088/0951-7715/30/1/61}.
\newblock URL \url{https://doi.org/10.1088/0951-7715/30/1/61}.

\bibitem[Glutsyuk et~al.(2014)Glutsyuk, Kleptsyn, Filimonov, and Schurov]{glutsyuk2014adjacency}
A.~Glutsyuk, V.~Kleptsyn, D.~Filimonov, and I.~Schurov.
\newblock On the adjacency quantization in an equation modeling the {J}osephson effect.
\newblock \emph{Functional Analysis and Its Applications}, 48\penalty0 (4):\penalty0 272--285, 2014.
\newblock \doi{10.1007/s10688-014-0070-z}.
\newblock URL \url{https://doi.org/10.1007/s10688-014-0070-z}.

\bibitem[Gr\"{u}ner and Zettl(1985)]{gruner1985charge}
G.~Gr\"{u}ner and A.~Zettl.
\newblock Charge density wave conduction: A novel collective transport phenomenon in solids.
\newblock \emph{Physics Reports}, 119\penalty0 (3):\penalty0 117--232, 1985.
\newblock \doi{10.1016/0370-1573(85)90073-0}.
\newblock URL \url{https://doi.org/10.1016/0370-1573(85)90073-0}.

\bibitem[Hatsuda(2020)]{hatsuda2020quasinormal}
Y.~Hatsuda.
\newblock Quasinormal modes of {K}err-de {S}itter black holes via the {H}eun function.
\newblock \emph{Classical and Quantum Gravity}, 38\penalty0 (2):\penalty0 025015, 2020.
\newblock \doi{10.1088/1361-6382/abc82e}.
\newblock URL \url{https://doi.org/10.1088/1361-6382/abc82e}.

\bibitem[Horta{\c{c}}su(2018)]{hortaccsu2018heun}
M.~Horta{\c{c}}su.
\newblock {H}eun functions and some of their applications in physics.
\newblock \emph{Advances in High Energy Physics}, 2018\penalty0 (1):\penalty0 8621573, 2018.
\newblock \doi{10.1155/2018/8621573}.
\newblock URL \url{https://doi.org/10.1155/2018/8621573}.

\bibitem[Ilyashenko et~al.(2011)Ilyashenko, Ryzhov, and Filimonov]{IRF}
Y.~Ilyashenko, D.~Ryzhov, and D.~Filimonov.
\newblock Phase-lock effect for equations modeling resistively shunted {J}osephson junctions and for their perturbations.
\newblock \emph{Functional Analysis and Its Applications}, 45\penalty0 (3):\penalty0 192--203, 2011.
\newblock \doi{10.1007/s10688-011-0023-8}.
\newblock URL \url{https://doi.org/10.1007/s10688-011-0023-8}.

\bibitem[Ishkhanyan(2016)]{ishkhanyan2016}
A.~Ishkhanyan.
\newblock Schr{\"o}dinger potentials solvable in terms of the confluent {H}eun functions.
\newblock \emph{Theoretical and Mathematical Physics}, 188\penalty0 (1):\penalty0 980--993, 2016.
\newblock \doi{10.1134/S0040577916070023}.
\newblock URL \url{https://doi.org/10.1134/S0040577916070023}.

\bibitem[Klimenko and Romaskevich(2014)]{klim-rom}
A.~Klimenko and O.~Romaskevich.
\newblock Asymptotic properties of {A}rnold tongues and {J}osephson eﬀect.
\newblock \emph{Moscow Mathematical Journal}, 14\penalty0 (2):\penalty0 367--384, 2014.

\bibitem[Likharev(1978)]{likh-ulr}
B.~Likharev, K.K.;~Ulrikh.
\newblock \emph{Systems with Josephson junctions: Basic Theory. [In Russian]}.
\newblock Izdat. MGU, Moscow, 1978.

\bibitem[Likharev(1985)]{lich-rus}
K.~Likharev.
\newblock \emph{Introduction to the dynamics of Josephson junctions. [In Russian]}.
\newblock Moscow, Nauka, 1985.

\bibitem[Likharev(1986)]{lich}
K.~Likharev.
\newblock \emph{Dynamics of Josephson junctions and circuits}.
\newblock Routledge, 1986.
\newblock \doi{10.1201/9781315141572}.
\newblock URL \url{https://doi.org/10.1201/9781315141572}.

\bibitem[Likharev(1971)]{ls}
V.~Likharev, K.K.;~Semenov.
\newblock Electrodynamical properties of superconducting point contacts. [in russian].
\newblock \emph{Radiotechnika i Electronika}, 16\penalty0 (11):\penalty0 2167, 1971.

\bibitem[Lisovyy and Naidiuk(2022)]{lisovyy2022perturbative}
O.~Lisovyy and A.~Naidiuk.
\newblock Perturbative connection formulas for {H}eun equations.
\newblock \emph{Journal of Physics A: Mathematical and Theoretical}, 55\penalty0 (43):\penalty0 434005, 2022.
\newblock \doi{10.1088/1751-8121/ac9ba7}.
\newblock URL \url{https://doi.org/10.1088/1751-8121/ac9ba7}.

\bibitem[Litvinov et~al.(2014)Litvinov, Lukyanov, Nekrasov, and Zamolodchikov]{litvinov2014classical}
A.~Litvinov, S.~Lukyanov, N.~Nekrasov, and A.~Zamolodchikov.
\newblock Classical conformal blocks and {P}ainlev{\'e} {VI}.
\newblock \emph{Journal of High Energy Physics}, 2014\penalty0 (7):\penalty0 1--20, 2014.
\newblock \doi{10.1007/JHEP07(2014)144}.
\newblock URL \url{https://doi.org/10.1007/JHEP07(2014)144}.

\bibitem[McCumber(1968)]{mccumber1968effect}
D.~E. McCumber.
\newblock Effect of {AC} impedance on {DC} voltage-current characteristics of superconductor weak-link junctions.
\newblock \emph{Journal of Applied Physics}, 39\penalty0 (7):\penalty0 3113--3118, 1968.
\newblock \doi{10.1063/1.1656743}.
\newblock URL \url{https://doi.org/10.1063/1.1656743}.

\bibitem[Mishra et~al.(2025)Mishra, Ryabov, and Maass]{mishra2025phase}
S.~Mishra, A.~Ryabov, and P.~Maass.
\newblock Phase locking and fractional {S}hapiro steps in collective dynamics of microparticles.
\newblock \emph{Physical Review Letters}, 134\penalty0 (10):\penalty0 107102, 2025.
\newblock \doi{10.1103/PhysRevLett.134.107102}.
\newblock URL \url{https://doi.org/10.1103/PhysRevLett.134.107102}.

\bibitem[Panghotra et~al.(2020)Panghotra, Raes, de~Souza~Silva, Cools, Keijers, Scheerder, Moshchalkov, and Van~de Vondel]{panghotra2020giant}
R.~Panghotra, B.~Raes, C.~C. de~Souza~Silva, I.~Cools, W.~Keijers, J.~E. Scheerder, V.~V. Moshchalkov, and J.~Van~de Vondel.
\newblock Giant fractional {S}hapiro steps in anisotropic {J}osephson junction arrays.
\newblock \emph{Communications Physics}, 3\penalty0 (1):\penalty0 53, 2020.
\newblock \doi{10.1038/s42005-020-0315-5}.
\newblock URL \url{https://doi.org/10.1038/s42005-020-0315-5}.

\bibitem[Pikovsky et~al.(2001)Pikovsky, Rosenblum, and Kurths]{pikovsky2001synchronization}
A.~Pikovsky, M.~Rosenblum, and J.~Kurths.
\newblock \emph{Synchronization: {A} {U}niversal {C}oncept in {N}onlinear {S}ciences}.
\newblock Cambridge University Press, 2001.
\newblock ISBN 9780521533522.
\newblock \doi{10.1017/CBO9780511755743}.

\bibitem[Reichhardt and Olson(2015)]{reichhardt2015shapiro}
C.~Reichhardt and J.~Olson.
\newblock Shapiro steps for skyrmion motion on a washboard potential with longitudinal and transverse ac drives.
\newblock \emph{Physical Review B}, 92\penalty0 (22):\penalty0 224432, 2015.
\newblock \doi{10.1103/PhysRevB.92.22443}.
\newblock URL \url{https://doi.org/10.1103/PhysRevB.92.22443}.

\bibitem[Renne and Polder(1974)]{renne1974some}
M.~J. Renne and D.~Polder.
\newblock Some analytical results for the resistively shunted {J}osephson junction.
\newblock \emph{Revue de physique appliqu{\'e}e}, 9\penalty0 (1):\penalty0 25--28, 1974.
\newblock \doi{10.1051/rphysap:019740090102500}.
\newblock URL \url{https://doi.org/10.1051/rphysap:019740090102500}.

\bibitem[Ronveaux and Arscott(1995)]{ronveaux1995heun}
A.~Ronveaux and F.~Arscott.
\newblock \emph{{H}eun's differential equations}.
\newblock Clarendon Press, 1995.
\newblock \doi{10.1093/oso/9780198596950.001.0001}.
\newblock URL \url{https://doi.org/10.1093/oso/9780198596950.001.0001}.

\bibitem[Shapiro(1963)]{shapiro1963josephson}
S.~Shapiro.
\newblock {J}osephson currents in superconducting tunneling: {T}he effect of microwaves and other observations.
\newblock \emph{Physical Review Letters}, 11\penalty0 (2):\penalty0 80, 1963.
\newblock \doi{10.1103/PhysRevLett.11.80}.
\newblock URL \url{https://doi.org/10.1103/PhysRevLett.11.80}.

\bibitem[Slavyanov and Lay(2000)]{slavyanov-lay2000}
S.~Slavyanov and W.~Lay.
\newblock \emph{Special Functions: A Unified Theory Based on Singularities}.
\newblock Oxford University Press, 2000.
\newblock \doi{10.1093/oso/9780198505730.001.0001}.
\newblock URL \url{https://doi.org/10.1093/oso/9780198505730.001.0001}.

\bibitem[Tertychniy(2006)]{tert2}
S.~Tertychniy.
\newblock The modelling of a {J}osephson junction and {H}eun polynomials, 2006.
\newblock URL \url{https://arxiv.org/abs/math-ph/0601064}.

\bibitem[Waldram and Wu(1982)]{waldram1982alternative}
J.~R. Waldram and P.~H. Wu.
\newblock An alternative analysis of the nonlinear equations of the current-driven {J}osephson junction.
\newblock \emph{Journal of Low Temperature Physics}, 47:\penalty0 363--374, 1982.
\newblock \doi{10.1007/BF00683738}.
\newblock URL \url{https://doi.org/10.1007/BF00683738}.

\end{thebibliography}

\end{document}